\documentclass[12pt,a4paper]{amsart}
\usepackage[english]{babel}
\usepackage[left=2cm,right=2cm,top=2cm,bottom=2cm]{geometry}
\usepackage[numbers,sort&compress]{natbib} 
\usepackage{enumitem}
\usepackage{hyperref}
\usepackage{graphicx}
\usepackage{mathtools}
\usepackage{cases}
\usepackage{bbm}
\usepackage{mathrsfs}  
\usepackage{float}
\usepackage{tikz}

\usepackage{xcolor}

\theoremstyle{definition}
\newtheorem{defi}{Definition}[section]

\newtheorem*{notation}{Notation}
\theoremstyle{plain}
\newtheorem{teo}[defi]{Theorem}
\newtheorem{lem}[defi]{Lemma}

\newtheorem{cor}[defi]{Corollary}

\newcommand{\N}{\mathbb{N}}
\newcommand{\R}{\mathbb{R}}

\newcommand{\se}{\subseteq}
\newcommand{\ceq}{\coloneqq}
\newcommand{\res} {\mathop{\hbox{\vrule height 7pt width .5pt depth 0pt \vrule height .5pt width 6pt depth 0pt}}\nolimits}
\newcommand{\leb}{\mathcal{L}}

\newcommand{\G}{\mathbb{G}}
\newcommand{\uno}{\mathbbm{1}}
\newcommand{\Z}{\mathbb{Z}}
\newcommand{\nl}{\left\|}
\newcommand{\nr}{\right\|}
\newcommand{\ve}{\varepsilon}

\title{Quantization of measures in Carnot groups}

\author{Marco Di Marco}
\author{Mikaela Iacobelli}

\address[M. Di Marco, M. Iacobelli]{ETH Z\"urich, Department of Mathematics, R\"amistrasse 101, 8092 Z\"urich, Switzerland.}

\email[M. Di Marco]{mdimarco@ethz.ch}
\email[M. Iacobelli]{mikaela.iacobelli@math.ethz.ch}

\subjclass[2020]{53C17, 41A60, 28A33, 49Q22, 41A65}
\keywords{Carnot groups, sub-Riemannian Geometry, optimal quantization, Zador's Theorem, Wasserstein distance}

\begin{document}

\begin{abstract}

We study optimal quantization of probability measures on Carnot groups equipped with a left-invariant homogeneous distance. We prove two main results. First, we establish a Zador-type asymptotic formula for the quantization error: after the natural rescaling, the error converges as the number of centers tends to infinity, and the exponent is determined by the homogeneous dimension of the group. The limiting constant is a Carnot cell constant, defined through the quantization problem on a reference cell. Second, we prove weak convergence of the empirical measures associated with optimal centers, and describe the limit in terms of the density of the absolutely continuous part of the measure. The proof combines a tiling of Carnot groups by exponential cubes with a sub-Riemannian version of Pierce’s lemma, which allows us to treat measures with non-compact support.\end{abstract}

\thanks{The authors are supported by SNSF Starting Grant TMSGI2\textunderscore226018. M. D. M. is a member of GNAMPA of INdAM. The authors warmly thank Camillo Brena for many precious discussions. Part of this work was written while M. D. M. was a guest at the Forschungsinstitut f\"ur Mathematik (FIM) in Z\"urich: he wishes to thank the FIM for the support, as well as for the pleasant and exceptionally stimulating atmosphere.}

\maketitle

\section{Introduction}

The term \emph{quantization of measures} classically refers to the approximation of a probability measure $\mu$ on $\R^n$, with respect to the $r$-th Wasserstein distance $W_r$, by probability measures supported on at most $N$ points. In other words, one would like to study the \emph{quantization error} functional defined by
\[
V_{N,r}(\mu)\ceq\inf_{\#\operatorname{spt}(\mu_N)\leq N} W_r(\mu,\mu_N)^r,
\]
where, as usual, the Wasserstein distance is defined as follows:
\[
W_r(\mu,\mu_N)^r \ceq \inf \left\{ \int_{\R^n \times \R^n} |x-y|^r\, d\gamma(x,y): (\pi_1)_\#\gamma=\mu, (\pi_2)_\#\gamma=\mu_N \right\}.
\]
This problem admits several equivalent formulations (see \cite{gl00}) and, in particular, it reduces to a problem over the support of $\mu_N$. In fact, $V_{N,r}$ admits also the equivalent definition
\[
V_{N,r}(\mu) \ceq \inf_{\alpha \se \R^n,\#\alpha \leq N} \int_{\R^n} \min_{a \in \alpha}|y-a|^r\, d\mu(y).
\]

Quantization of measures also has applications in signal processing, materials science, economics, and numerical analysis. We briefly recall the history and applications of the problem before turning to the Carnot-group setting. The quantization problem originated in signal processing, where the aim is to efficiently compress analog signals, such as sound or images, into digital signals taking values in a finite set. The study of scalar quantization, i.e. quantization on $\R$, started in the 1940s with the pioneering works \cite{b48,ops48}. Later, starting from the 1970s, the quantization problem on $\R^n$ was studied by different authors \cite{g79,z82,gl00,bw82}, under the name of vector quantization. For a complete overview of quantization from the viewpoint of signal processing and information theory, we refer the reader to \cite{gg92,gn02}, and for an introduction to quantization and its applications in numerics, to \cite{g15}.

We also mention other applications in materials science \cite{bpt14} and mathematical models in economics \cite{bs09,s15}.

The quantization problem has also been studied under different names, such as the optimal location problem \cite{bjr02,bjm11,br15,bc21}, centroidal Voronoi tessellations \cite{qvg99,qw05}, or the problem of k-means clustering in statistics \cite{m67}. An associated problem is empirical quantization, or uniform quantization, studied both in the deterministic setting \cite{c18,k20,mss21} and in the stochastic setting, also referred to as random matching, \cite{hs82,z85,gl02,dss13,ts15,ag19,agt19,ast19,bc20,bccdss21}.

Let us also mention that some recent different approaches in the study of quantization include the use of calculus of variations/$\Gamma$-convergence \cite{bjm11,bc21} and gradient flow \cite{cgi15,cgi18,i18,i19,ips19}.

In recent years, the quantization problem has also been studied on Riemannian manifolds, see for instance \cite{k12,ai25,i16,g04,g01,lbp19,lbp19bis}.

For further details on the basics of the theory of quantization of measures, we refer the reader to the monograph \cite{gl00} and the references therein. 

A central result in the theory of quantization is Zador's Theorem (Theorem \ref{teo_classic} below), which provides a precise expression for the asymptotics of the quantization error $V_{N,r}$. Zador's theorem was first proved in a weaker version in \cite{z82} and later improved by several authors, see \cite{bw82,gl00}.

\begin{teo}\label{teo_classic}
    Let $r\geq 1$, $\delta>0$, and $\mu$ be a probability measure on $\R^n$ with finite $(r+\delta)$-moment, that is,
    \[
    \int_{\R^n}|x|^{r+\delta}\, d\mu(x)<+\infty.
    \]
    Define
    \[
    Q_r([0,1]^n) \ceq \inf_{N \geq 1}N^{\frac{r}{n}}V_{N,r}(\uno_{[0,1]^n}\leb^n),
    \]
    where $\uno_{[0,1]^n}$ denotes the indicator function of the $n$-dimensional cube $[0,1]^n \se \R^n$. Then $Q_r([0,1]^n)>0$. Let $\mu=\mu_{\mathrm{ac}}+\mu_{\mathrm{s}}$ be the Lebesgue decomposition\footnote{With respect to the $n$-dimensional Lebesgue measure $\leb^n$.} of $\mu$, and set $h\ceq d\mu_{\mathrm{ac}}/d\leb^n$. Then
\[
\lim_{N \to +\infty}N^{\frac{r}{n}}V_{N,r}(\mu)=Q_r([0,1]^n)\nl h \nr_{L^{\frac{n}{n+r}}}.
\]
\end{teo}
Zador's Theorem was later extended to compactly supported probability measures (see \cite{g04,k12}) and to non-compactly supported probability measures (see \cite{i16,ai25}) on Riemannian manifolds. Let us also mention the recent generalization of a weaker version of Zador's Theorem in the broader context of metric measure spaces \cite{a25}.

The classical Euclidean and Riemannian proofs proceed as follows. In the Euclidean case, one first solves the cell problem for the uniform measure on a cube, then treats finite unions of cubes by optimizing the allocation of centers among the cubes, and finally passes to general densities by approximation and Pierce's Lemma \cite[Theorem~6.2 and Lemma~6.6]{gl00}. The density $h$ enters through this allocation principle, producing the factor $h^{n/(n+r)}$; Riemannian versions localize the same strategy in sufficiently small charts, where the metric and the volume are asymptotically Euclidean \cite{g04,k12,i16,ai25}. This localization step is not automatic in the Carnot setting.
\par

Recent metric-measure results provide a useful general framework for the quantization problem, but they do not directly yield the precise asymptotic formula proved here. In particular, \cite{a25} proves a weaker Zador-type result for general metric measure spaces in terms of Hausdorff densities, and a full Zador theorem for appropriate $m$-rectifiable measures in Euclidean space. In the Carnot-group setting, such results can detect the natural scaling associated with the homogeneous dimension $Q$, but they do not identify the exact asymptotic constant, nor do they prove the existence of the limit with the Carnot cell constant appearing in Theorem \ref{teo_intro}. The aim here is to solve this local cell problem in the non-Euclidean tangent geometry of the group; the proof uses the intrinsic pavage by left-translated and homogeneously dilated exponential cubes, together with a Carnot-group version of Pierce's lemma.
\par

Analysis on sub-Riemannian manifolds\footnote{For a complete introduction to sub-Riemannian geometry, we refer the reader to the monographs \cite{abb,cdptbook,montgomery,LDlibro} and the references therein.} has developed in part by testing which classical results extend to this setting. This paper begins the study of quantization of measures in the sub-Riemannian setting.

 Carnot groups (see \cite{ld17} for a primer) are the basic non-Euclidean models in this setting. We postpone precise definitions to Section \ref{sec_prelim}; here we only recall that a Carnot group $(\G,\cdot)$ (see Definition \ref{def_carnot}) is a nilpotent, connected, simply connected Lie group (diffeomorphic to $\R^n$, where $n$ is the topological dimension of $\G$) such that its Lie algebra is stratified. Every Carnot group can be endowed with a left-invariant and homogeneous distance $d$ and with a one-parameter family of group isomorphisms $(\delta_\lambda)_{\lambda>0}$, called dilations. The Hausdorff (homogeneous) dimension of the metric space $(\G,d)$ is greater\footnote{Or equal, in the trivial case when $\G$ is a commutative group.} than the topological dimension of $\G$. Thus it is not a priori clear whether the Euclidean/Riemannian Zador formula survives, or which dimension governs the asymptotic rate. The natural dilations act with different weights on the layers of the Lie algebra, so the metric dimension is the homogeneous dimension \(Q\), not the topological dimension \(n\). The classical argument therefore does not determine whether the asymptotic exponent is \(n\), \(Q\), or also depends on the group law.

A second obstruction is that the known Riemannian proofs cannot be transferred by simply working in local coordinates. Although a Carnot group is a smooth manifold, and can be identified with $\R^n$ through exponential coordinates, the homogeneous distance is not locally Riemannian in these coordinates: small metric balls have volume of order $\rho^Q$, not $\rho^n$, and the natural blow-ups are anisotropic dilations rather than Euclidean rescalings. Thus the usual reduction to Euclidean quantization in almost-isometric Riemannian charts would lead to the wrong local model; the relevant local geometry is the Carnot group itself, with its homogeneous distance and Haar/Lebesgue measure, and this requires replacing Euclidean self-similar cubes by the pavage property used below.
\par

Our first main result is a sub-Riemannian version of Zador's Theorem (Theorem \ref{teo_intro}) for every Carnot group equipped with a left-invariant homogeneous distance, including the Heisenberg groups. It shows that the Euclidean exponent $n$ is replaced by the homogeneous dimension $Q$, while only the absolutely continuous part of the measure contributes to the limiting constant.

\begin{teo}\label{teo_intro}
 Let $\G$ be a Carnot group of topological dimension $n$ and homogeneous dimension $Q$ equipped with a left-invariant homogeneous distance $d$. Let $C$ be the Euclidean cube $C \ceq [-1/2,1/2)^n$, where, here and in the following, we identify $\G$ with $\R^n$ by means of exponential coordinates. Let $r \geq 1$, $\delta>0$, and $\mu$ be a probability measure on $\G$ with finite $(r+\delta)$-moment, that is,
 \[
 \int_{\G}d(x,0)^{r+\delta}\, d\mu(x)<+\infty.
    \]
    Define
    \[
    Q_r(C) \ceq \inf_{N \geq 1}N^{\frac{r}{Q}}V_{N,r}(\uno_{C}\leb^n).
    \]
    Then $Q_r(C)>0$. Let $\mu=\mu_{\mathrm{ac}}+\mu_{\mathrm{s}}$ be the Lebesgue decomposition\footnote{With respect to the $n$-dimensional Lebesgue measure $\leb^n$.} of $\mu$, and set $h\ceq d\mu_{\mathrm{ac}}/d\leb^n$. Then
\begin{equation}\label{eq_zadcarnot}
\lim_{N \to +\infty}N^{\frac{r}{Q}}V_{N,r}(\mu)=Q_r(C)\nl h \nr_{L^{\frac{Q}{Q+r}}}.
\end{equation}
\end{teo}

Our second main result identifies the asymptotic distribution of optimal centers. In analogy with the Euclidean theory \cite[Section~7.2]{gl00}, where the limiting density is proportional to \(h^{n/(n+r)}\), the Carnot analogue involves the density \(h^{Q/(Q+r)}\).

\begin{teo}[Empirical measures of optimal centers]\label{teo_empirical_centers}
Let $\mu$ satisfy the assumptions of Theorem \ref{teo_intro}, and assume that $\mu_{\mathrm{ac}}(\G)>0$. Let
\[
h\ceq \frac{d\mu_{\mathrm{ac}}}{d\leb^n}.
\]
For every $N\geq2$, let $\alpha_N\in C_{N,r}(\mu)$, where \(C_{N,r}(\mu)\) denotes the family of \(N\)-optimal sets. Then $\#\alpha_N=N$, and, setting
\[
\nu_N\ceq \frac1N\sum_{a\in\alpha_N}\delta_a,
\]
where $\delta_a$ denotes the Dirac mass at $a$, one has
\[
\nu_N
\rightharpoonup
\frac{h^{\frac{Q}{Q+r}}}{\int_\G h^{\frac{Q}{Q+r}}\,d\leb^n}\,\leb^n
\qquad\text{as }N\to+\infty .
\]
\end{teo}
The proof uses the same localization as in Lemma \ref{lem_step2}, together with the equality case in the allocation inequality, and is given at the end of Section \ref{sec_qom}.
\par

The proof of the classical Zador Theorem (see \cite{gl00}), as well as some of its aforementioned generalizations, relies on the use of a self-similar tessellation of $\R^n$ made up of cubes.
 In a Carnot group, the exponential cube $C$ is not self-similar with respect to the group law and the homogeneous dilations. The substitute we use is the following \emph{pavage} property (Lemma \ref{lem_intropav}): left-translated and homogeneously dilated copies of $C$ still tile the group and can be used to approximate the Euclidean cube $C$. This provides the multiscale localization needed in place of Euclidean cubes.

\begin{lem}[{\cite[Lemma 2.4]{fgn05}}]\label{lem_intropav}
    For every $\lambda>0$, the family $ \{ \delta_\lambda(p \cdot C): p \in \Z^n \}$ is a pavage of $\G$, i.e.,
    \[
    \bigsqcup_{p \in \Z^n}\delta_\lambda(p \cdot C)=\G.
    \]
\end{lem}
In other words, we will approximate the Euclidean cube $C$ with a sequence of sets $(C_\lambda)_{\lambda>0}$ that are finite disjoint unions of suitably left-translated and homogeneously dilated copies of $C$.

The fact that one can construct pavages of Carnot groups by using left-translations and dilations of the Euclidean cubes has been exploited for many purposes, see for instance \cite{bw03,ft02,fgn05,ddmm20}.

We conclude the introduction by outlining the structure of the rest of the paper. 
In Section \ref{sec_prelim} we present the definitions and some preliminary results. In Section \ref{sec_pierce} we present a sub-Riemannian version of the so-called Pierce's Lemma (\cite{p70}, see also \cite{gl00}), Theorem \ref{teo_intropierce} below.

\begin{teo}\label{teo_intropierce}
 Let $\G$ be a Carnot group of topological dimension $n$ and homogeneous dimension $Q$ equipped with a left-invariant homogeneous distance $d$. Let $r \geq 1$, $\delta>0$, and $\mu$ be a probability measure on $\G$ with finite $(r+\delta)$-moment. Then there exists a constant $c=c(r,\delta,Q)>0$ such that, for every $N \geq 2$, 
\[
V_{N,r}(\mu) \leq c N^{-\frac{r}{Q}}\left(1+\int_\G d(x,0)^{r+\delta}\, d\mu(x)\right).
\]
\end{teo}
Theorem \ref{teo_intropierce} provides a uniform integral upper bound on the quantization error and is used below to prove the sub-Riemannian Zador theorem for measures without compact support. In Section \ref{sec_qom} we prove Theorem \ref{teo_intro}, first in several special cases and then in full generality, and then prove Theorem \ref{teo_empirical_centers}.

\section{Notation and preliminary results}
\label{sec_prelim}

\begin{defi}\label{def_carnot}
   A \emph{Carnot} (or \emph{stratified}) \emph{group} $\G$ of \emph{step} $s$ is a nilpotent, connected, simply connected Lie group such that its Lie algebra $\mathfrak{g}$ admits a stratification of step $s$, i.e., there exist linear subspaces $V_1,\dots,V_s \se \mathfrak g$ such that, for every $j=1,\dots, s-1$,
    \[
    \mathfrak g= V_1 \oplus \cdots \oplus V_s,\qquad [V_1,V_j] = V_{j+1}, \qquad V_s \neq \{ 0 \}, \quad \text{and} \quad [V_s,\mathfrak g]=\{0\}.
    \]
    The exponential map $\exp:\mathfrak g \to \G$ is a diffeomorphism, and, given a basis $X_1,\dots,X_n$ adapted to the stratification and ordered so that the corresponding blocks span $V_1,\dots,V_s$, in the following we will identify $\G$ with $\R^n$ by means of exponential coordinates:
    \[
    \R^n \ni x =(x_1,\dots,x_n) \leftrightarrow \exp (x_1X_1+\cdots x_nX_n) \in \G.
    \]
    If we set $m_i \ceq \dim V_i$, then $\sum_{i=1}^s m_i=n$, which is the \emph{topological dimension} of $\G$, and $Q \ceq \sum_{i=1}^s i m_i$, which is the \emph{homogeneous dimension} of $\G$. By further identifying $\G \equiv \R^{m_1} \times \cdots \times \R^{m_s}$, we define, for $\lambda >0$, the \emph{homogeneous dilations}
    \[
    \delta_\lambda(x_1,\dots,x_s)=(\lambda x_1,\dots,\lambda^s x_s), \qquad x_i \in \R^{m_i}.
    \]
  The usual Lebesgue measure $\leb^n$ is, up to a multiplicative constant, the Haar measure of $\G$, and it is left-invariant (with respect to the group operation $\cdot$) and $Q$-homogeneous with respect to homogeneous dilations, i.e.,
  \[
  \leb^n(A)=\leb^n(p \cdot A), \qquad \leb^n(\delta_\lambda A)=\lambda^Q\leb^n(A),
  \]
  for every $p \in \G, A \se \G, \lambda>0$. We denote by $d$ a left-invariant and homogeneous distance\footnote{Such as the Carnot-Carathéodory distance.} on $\G$, i.e., a distance $d:\G \times \G \to [0,+\infty)$ such that
  \[
  d(p,q)=d(r \cdot p,r \cdot q), \qquad d(\delta_\lambda p,\delta_\lambda q)=\lambda d(p,q),
  \]
   for every $p,q,r \in \G, \lambda>0$. For every $p \in \G$ and $\rho>0$ we denote by $B(p,\rho)$ the open ball
   \[
   B(p,\rho) \ceq \{q \in \G: d(p,q)<\rho\}.
   \]
\end{defi}

Let us also recall the following \emph{pavage} property of the Euclidean cube $C= [-1/2,1/2)^n \se \G$ which will be extensively used in the sequel. For further details we refer the reader to \cite[Section 2.3]{fgn05}.
\begin{lem}\label{lem_pavage}
    For every $\lambda>0$, the family $ \{ \delta_\lambda(p \cdot C): p \in \Z^n \}$ is a pavage of $\G$, i.e.,
    \[
    \bigsqcup_{p \in \Z^n}\delta_\lambda(p \cdot C)=\G.
    \]
\end{lem}

\begin{center}
    \begin{minipage}{0.98\linewidth}
    \centering
    \begin{tikzpicture}[x=1cm,y=1cm, scale=1.22, transform shape, line cap=round, line join=round, every node/.style={font=\footnotesize}]
    \tikzset{
    cell/.style={draw=black!55, fill=black!3, line width=0.32pt},
    cellstrong/.style={draw=black!70, fill=black!6, line width=0.36pt},
    maincell/.style={draw=blue!65!black, fill=blue!7, line width=0.55pt},
    boundary/.style={draw=black!75, line width=0.58pt},
    title/.style={anchor=west, font=\footnotesize, text=black!85},
    note/.style={text=black!70, font=\scriptsize}
    }
    \begin{scope}[shift={(0,0)}]
    \node[title] at (-1.55,1.65) {(a) Euclidean tiling};
    \foreach \i in {-2,-1,0}{
        \foreach \j in {-1,0,1}{
            \draw[cell] (0.55*\i,0.55*\j) rectangle ++(0.55,0.55);
        }
    }
    \draw[maincell] (0,0) rectangle ++(0.55,0.55);
    \node[blue!65!black, font=\scriptsize] at (0.275,0.275) {$C$};
    \end{scope}
    \begin{scope}[shift={(3.45,0)}]
    \node[title] at (-1.55,1.65) {(b) Left translations};
    \foreach \i in {-2,-1,0}{
        \foreach \j in {-1,0,1}{
            \pgfmathsetmacro{\x}{0.52*\i+0.13*\j}
            \pgfmathsetmacro{\y}{0.52*\j}
            \draw[cellstrong] (\x,\y) -- ++(0.52,0.04) -- ++(0.13,0.52) -- ++(-0.52,-0.04) -- cycle;
        }
    }
    \draw[maincell] (0,0) -- ++(0.52,0.04) -- ++(0.13,0.52) -- ++(-0.52,-0.04) -- cycle;
    \end{scope}
    \begin{scope}[shift={(7.05,0)}]
    \node[title] at (-1.55,1.65) {(c) Inner approximation};
    \fill[black!10] (-1.35,-1.0) rectangle (1.35,1.0);
    \fill[white] (-1.02,-0.72) rectangle (1.02,0.72);
    \draw[boundary] (-1.35,-1.0) rectangle (1.35,1.0);
    \node[anchor=north east] at (1.28,0.93) {$C$};
    \begin{scope}
    \clip (-1.02,-0.72) rectangle (1.02,0.72);
    \foreach \i in {-3,-2,-1,0,1,2}{
        \foreach \j in {-2,-1,0,1,2}{
            \pgfmathsetmacro{\x}{0.31*\i+0.06*\j}
            \pgfmathsetmacro{\y}{0.31*\j}
            \draw[cellstrong] (\x,\y) -- ++(0.31,0.03) -- ++(0.06,0.31) -- ++(-0.31,-0.03) -- cycle;
        }
    }
    \end{scope}
    \draw[maincell] (-0.01,-0.01) -- ++(0.31,0.03) -- ++(0.06,0.31) -- ++(-0.31,-0.03) -- cycle;
    \node[note, align=center] at (0,-1.18) {$C_\lambda\subset C$};
    \end{scope}
    \end{tikzpicture}

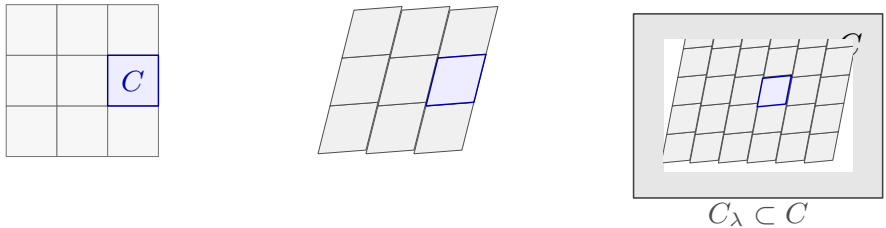
    {\small \textbf{Figure.} Schematic pavage. Left-translated and dilated copies of \(C\) tile the group; the cells contained in \(C\) form the inner approximation \(C_\lambda\). The drawings are projections.}
    \end{minipage}
\end{center}

\begin{notation}
    In the following, unless otherwise specified, $\G$ will denote a Carnot group of topological dimension $n$ and homogeneous dimension $Q$ which is equipped with a left-invariant and homogeneous distance $d$. Moreover, we will assume $r \geq 1$ and we will use $\mu$ to denote a probability measure on $\G$ with finite $r$-moment, that is
\[
\int_\G d(x,0)^r\, d\mu(x)<+\infty.
\]
Moreover, we will use $\delta$ to denote a positive real number, $[\cdot]$ to denote the integer part function and $C$ to denote the Euclidean cube $[-1/2,1/2)^n$, where, here and in the following, we identify $\G$ with $\R^n$ by means of exponential coordinates. 
\end{notation}

\begin{defi}
Let $A \se \G$ be a Borel set with $0<\leb^n(A)<+\infty$, let $N \in \N$, and let $\mu$ be a measure on $\G$ such that $\mu(\G)$ is positive and finite. Then we define:
\begin{itemize}
\item the \emph{uniform measure} $U(A)$ of $A$ as
\[
U(A) \ceq\frac{\uno_{A}\leb^{n}}{\leb^{n}(A)};
\]
\item the \emph{$N$-th quantization error of order $r$} $V_{N,r}(\mu)$ as
    \[
    V_{N,r}(\mu) \ceq \inf_{\alpha \se \G,1\leq\#\alpha \leq N} \int_{\G} \min_{a \in \alpha}d(a,y)^r\,d\mu(y);
    \]
\item the set of \emph{$N$-optimal centers for $\mu$ of order $r$} $C_{N,r}(\mu)$: we say that $\alpha \se \G$ with $1\leq\#\alpha \leq N$ belongs to $C_{N,r}(\mu)$ if
    \[
    V_{N,r}(\mu)=\int_{\G} \min_{a \in \alpha}d(a,y)^r\,d\mu(y);
    \]
\item the \emph{normalized $N$-th quantization error of order $r$ for $A$} $M_{N,r}(A)$ as
\[
M_{N,r}(A) \ceq \frac{V_{N,r}(U(A))}{\leb^n(A)^{\frac{r}{Q}}};
\]
\item and
\[
Q_r(A) \ceq  \inf_{N \geq 1}{N^{\frac{r}{Q}}}M_{N,r}(A).
\]
\end{itemize}
\end{defi}
In the following lemma we collect the basic properties of the above functionals that will be useful in the sequel. The proof of the lemma below follows immediately from the definitions above.\footnote{For a proof of the Euclidean counterpart of Lemma \ref{lem_basicpro}, see for instance \cite{gl00}.}
\begin{lem}\label{lem_basicpro}
   The following properties hold for every Borel set $A \se \G$ with $0<\leb^n(A)<+\infty$, $p \in \G$, $\lambda>0$, and $N \in \N$.
   \begin{itemize}
          \item[(i)] $V_{N,r}(U(p \cdot A))=V_{N,r}(U(A))$, \quad $V_{N,r}(U(\delta_\lambda A))=\lambda^rV_{N,r}(U(A))$;
       \item[(ii)] $p \cdot C_{N,r}(\mu)=C_{N,r}((p \cdot)_\#\mu)$, \quad $\delta_\lambda(C_{N,r}(\mu))=C_{N,r}((\delta_\lambda)_\#\mu)$; equivalently, $p\cdot C_{N,r}(\mu)$ denotes $\{p\cdot\alpha:\alpha\in C_{N,r}(\mu)\}$ and similarly for $\delta_\lambda(C_{N,r}(\mu))$.
       \item[(iii)] $M_{N,r}(p \cdot A)=M_{N,r}(A)$, \quad $M_{N,r}(\delta_\lambda A)=M_{N,r}(A)$;
       \item[(iv)] if $\mu=\sum_{\ell=1}^m s_\ell \mu_\ell$ with $\sum_{\ell=1}^m s_\ell=1$, $0 \leq s_\ell \leq 1$, and $\mu_\ell$ are probability measures on $\G$ with finite $r$-moment, then
       \begin{itemize}
           \item[(iv.A)] $V_{N,r}$ is concave, i.e.,
\[
V_{N,r}(\mu) \geq \sum_{\ell=1}^m s_\ell V_{N,r}(\mu_\ell).
\]
       \item[(iv.B)] if $N_\ell \in \N$ and $\sum_{\ell=1}^m N_\ell \leq N$, then
       \[
       V_{N,r}(\mu) \leq \sum_{\ell=1}^m s_\ell V_{N_\ell,r}(\mu_\ell).
       \]
              \end{itemize}
\item[(v)] if $\mu=s\mu_1+(1-s)\mu_2$ with $0\leq s \leq 1$, $\mu_1,\mu_2$ probability measures on $\G$ with finite $r$-moment, and
\[
\lim_{N \to +\infty}N^{\frac{r}{Q}}V_{N,r}(\mu_1) = c \in [0,+\infty),
\]
then
\begin{itemize}
    \item[(v.A)]
    \begin{align*}
\liminf_{N \to +\infty} N^{\frac{r}{Q}}V_{N,r}(\mu)
&\geq sc+(1-s)\liminf_{N \to +\infty}N^{\frac{r}{Q}}V_{N,r}(\mu_2),\\
\limsup_{N \to +\infty} N^{\frac{r}{Q}}V_{N,r}(\mu)
&\leq s(1-\ve)^{-\frac{r}{Q}}c+(1-s)\ve^{-\frac{r}{Q}}\limsup_{N \to +\infty}N^{\frac{r}{Q}}V_{N,r}(\mu_2),
    \end{align*}
    for every $0<\ve<1$;
    \item[(v.B)] if
    \[
    \lim_{N \to +\infty}N^{\frac{r}{Q}}V_{N,r}(\mu_2)=0,
    \]
    then
    \[
    \lim_{N \to +\infty}N^{\frac{r}{Q}}V_{N,r}(\mu)=sc.
    \]
\end{itemize}
   \end{itemize}
\end{lem}
We now prove some preliminary results, namely Lemmas \ref{lem_existence} and \ref{lem_card}, which ensure that, for every $N \in \N$, there exists an $N$-optimal set, and that, under the support condition in Lemma \ref{lem_card}, every such set has cardinality exactly $N$; Theorem \ref{teo_wp} ensures that the quantization problem is well posed, that is, the quantization error goes to $0$ as the number of approximating points goes to infinity.

\begin{lem}\label{lem_existence}

    For every $N \in \N$ one has $C_{N,r}(\mu) \neq \emptyset$.
\end{lem}
\begin{proof}
We use the equivalent formulation
\[
V_{N,r}(\mu)=\inf_{\nu:\,1\leq\#\operatorname{spt}(\nu)\leq N} W_r(\mu,\nu)^r,
\]
which follows from the usual reduction of optimal transport to the nearest-point projection onto the support of $\nu$.
For every $N \in \N$ let $(\mu^k_N)_{k \in \N}$ be a sequence of probability measures such that $\#\operatorname{spt}(\mu_N^k) \leq N$ which is minimizing (with respect to the Wasserstein distance $W_r$), that is, 
\[
W_r(\mu,\mu_{N}^k)^r \xrightarrow{k \to +\infty}V_{N,r}(\mu).
\]
We claim that, up to a subsequence, there exists a probability measure $\bar \mu_N$ such that $\mu_N^k \rightharpoonup \bar \mu_N$. By Prokhorov's Theorem the claim holds if, for every $\ve>0$, there exists a compact set $K_\ve \se \G$ such that $\mu_{N}^k(\G \setminus K_\ve)<\ve$ for every $k \in \N$ (i.e., the sequence is tight). Since $\mu$ has finite $r$-moment, we have $W_r(\mu,\delta_0)<+\infty$. Moreover, since
\[
W_r(\mu,\mu_N^k)^r \xrightarrow{k \to +\infty}V_{N,r}(\mu),
\]
there exists a constant $C>0$ such that
\[
W_r(\mu,\mu_N^k)\leq C
\]
for every $k \in \N$. Hence, by the triangle inequality,
\[
W_r(\mu_N^k,\delta_0)\leq W_r(\mu_N^k,\mu)+W_r(\mu,\delta_0)\leq C+W_r(\mu,\delta_0)
\]
for every $k \in \N$. Therefore there exists a constant $c>0$, independent of $k$, such that
\[
W_r(\mu_N^k,\delta_0)^r\leq c
\]
for every $k \in \N$. Let $\ve>0$, and $R>0$ to be determined later and define $K_\ve \ceq \overline{B(0,R)}$. Then for every $k \in \N$ we have
\[
c \geq W_r(\mu_N^k,\delta_{0})^r \geq \int_{\G \setminus K_\ve}d(x,0)^r\,d \mu_N^k(x) \geq R^r \mu_N^k(\G \setminus K_\ve).
\]
Choosing $R\geq (c/\ve)^{1/r}$ is enough to prove the claim. Let us prove that $\#\operatorname{spt}(\bar \mu_N) \leq N$. We argue by contradiction: assume $\#\operatorname{spt}(\bar \mu_N) > N$; then there exist $N+1$ distinct points $\{x_1,\dots,x_{N+1}\} \se \operatorname{spt}(\bar \mu_N) $. The latter implies there exists $\rho>0$ such that the open balls $\{B(x_1,\rho),\dots,B(x_{N+1},\rho)\}$ are disjoint and $\bar \mu_N(B(x_i,\rho))>0$ for every $1 \leq i \leq N+1$. By Portmanteau's Theorem we have 
\[
\bar \mu_N(B(x_i,\rho)) \leq \liminf_{k \to +\infty} \mu^k_N(B(x_i,\rho)),
\]
the latter implying that for every $1 \leq i \leq N+1$ there exists $k_i \in \N$ such that, for every $k \geq k_i$, one has
\[
\mu_N^k(B(x_i,\rho))>0,
\]
obtaining that, whenever $k \geq  \max_{1 \leq i \leq N+1}k_i$, $\#\operatorname{spt}(\mu_N^k) > N$, obtaining a contradiction. Hence $\#\operatorname{spt}(\bar \mu_N) \leq N$ and
\begin{equation}\label{eq_boundsup}
W_r(\mu,\bar \mu_N)^r \geq V_{N,r}(\mu).
\end{equation}
By the lower semicontinuity of the Wasserstein distance $W_r$ we infer
\[
W_r(\mu,\bar \mu_N)^r \leq \liminf_{k \to +\infty} W_r(\mu,\mu^k_N)^r=V_{N,r}(\mu).
\]
Combining the latter with \eqref{eq_boundsup} we obtain $W_r(\mu,\bar \mu_N)^r =V_{N,r}(\mu)$. Let $\alpha\ceq \operatorname{spt}(\bar\mu_N)$. Since $\#\alpha\leq N$, by definition of $V_{N,r}$ we have
\[
V_{N,r}(\mu)\leq \int_\G \min_{a\in\alpha} d(a,y)^r\,d\mu(y).
\]
On the other hand, since $\bar\mu_N$ is supported on $\alpha$,
\[
W_r(\mu,\bar\mu_N)^r\geq \int_\G \min_{a\in\alpha} d(a,y)^r\,d\mu(y).
\]
Therefore
\[
V_{N,r}(\mu)=\int_\G \min_{a\in\alpha} d(a,y)^r\,d\mu(y),
\]
that is, $\alpha\in C_{N,r}(\mu)$.
\end{proof}
\begin{lem}\label{lem_card}
    For every $2 \leq N \in\N$ such that $\#\operatorname{spt}(\mu)>N-1$ and for every $\alpha \in C_{N,r}(\mu)$ one has $\#\alpha=N$. 
\end{lem}
\begin{proof}
    We just need to prove that, if $\#\operatorname{spt}(\mu)>N-1$, then 
    \[
    V_{N,r}(\mu)<V_{N-1,r}(\mu).
    \]
    Let $\beta \in C_{N-1,r}(\mu)$. Since $\#\beta \leq N-1$, we can choose $x \in \operatorname{spt}(\mu) \setminus \beta$ and define $\zeta \ceq \min_{b \in \beta}d(x,b)$. Since $x \notin \beta$ and $\beta$ is finite, we have $\zeta>0$. Since $x \in \operatorname{spt}(\mu)$, we have
    \begin{equation}\label{eq_zeta4}
    \mu(B(x,\tfrac{\zeta}{4}))>0.
    \end{equation}
    For every $y \in B(x,\tfrac{\zeta}{4})$ we have, by the triangle inequality,
    \[
    \min_{b \in \beta}d(y,b) \geq \min_{b \in \beta}d(x,b) -d(x,y) \geq \zeta -\frac{\zeta}{4}=\frac{3\zeta}{4}.
    \]
    and
    \[
    \min_{b \in \beta \cup \{x\}} d(y,b) \leq d(y,x)<\frac{\zeta}{4}.
    \]
    Therefore for every $y \in B(x,\tfrac{\zeta}{4})$ we have the following inequality
    \[
    \min_{b \in \beta \cup \{x\}} d(y,b)^r \leq \left(  \frac{\zeta}{4}  \right) ^r< \left(  \frac{3\zeta}{4}  \right) ^r \leq     \min_{b \in \beta}d(y,b)^r.
    \]
    Recalling \eqref{eq_zeta4} and integrating the latter, one obtains the strict inequality
    \[
V_{N,r}(\mu) \leq \int_\G   \min_{b \in \beta \cup \{x\}} d(y,b)^r\,d\mu(y) <  \int_\G \min_{b \in \beta}d(y,b)^r\,d\mu(y)=V_{N-1,r}(\mu
 ),
    \]
    concluding the proof.
\end{proof}
\begin{teo}\label{teo_wp}
    Let $\mu$ be a probability measure on $\G$ with finite $r$-moment. Then
    \[
   \lim_{N \to +\infty} V_{N,r}(\mu)=0.
    \]
\end{teo}
\begin{proof}
    Let $\{a_1,a_2,\dots\}$ be any enumeration of $\mathbb Q^{n}$ with $a_1=0$. For $\ve >0$ the family of open balls $B(a_k,(\ve/2)^\frac{1}{r})$ is a covering of $\G$. Therefore one can find a countable partition $(A_k)_{k \in \N}$ of $\G$ such that $A_k \se B(a_k,(\ve/2)^\frac{1}{r})$. Since $\mu$ has finite $r$-th moment there exists $\bar N \in \N$ such that 
    \[
   \sum_{k > \bar N} \int_{A_k} d(x,0)^rd\mu(x) \leq \frac{\ve}{2}.
    \]
    Then we conclude by observing
    \[
    V_{\bar N,r}(\mu) \leq\sum_{k \in \N}\int_{A_k} \min_{1 \leq i \leq \bar N} d(a_i,x)^rd\mu(x).
    \]
    For every $1\leq k\leq \bar N$ and every $x\in A_k$, we have
    \[
    \min_{1 \leq i \leq \bar N} d(a_i,x)^r \leq d(a_k,x)^r<\frac{\ve}{2},
    \]
    while for every $k>\bar N$ and every $x\in A_k$, since $a_1=0$, we have
    \[
    \min_{1 \leq i \leq \bar N} d(a_i,x)^r \leq d(a_1,x)^r=d(0,x)^r.
    \]
    Therefore
    \[
    V_{\bar N,r}(\mu) \leq \sum_{k=1}^{\bar N} \int_{A_k} \frac{\ve}{2}\,d\mu(x)+\sum_{k > \bar N} \int_{A_k}d(0,x)^rd\mu(x) \leq \frac{\ve}{2}\mu(\G)+\frac{\ve}{2}=\ve.
    \]
\end{proof}

\section{Pierce's Lemma in Carnot groups}\label{sec_pierce}
Before proving a sub-Riemannian version of Pierce's Lemma (Theorem \ref{teo_pierce}), we recall the definition of covering numbers and a preliminary result about the covering number of balls in Carnot groups.

\begin{defi}
    Let $A$ be a non-empty precompact subset of $\G$: the \emph{covering number} $N(A,\rho)$ of $A$ of radius $\rho>0$ is the smallest number of open balls of radius $\rho$ that cover $A$, i.e.,
    \[
    N(A,\rho)\ceq \min \left\{ N \in \N: \exists \{p_i\}_{i=1}^N \se \G, \ A \se \bigcup_{i=1}^N B(p_i,\rho)  \right\}.
    \]
\end{defi}
The following covering estimate follows from the standard maximal-packing argument; compare \cite[Lemma~2.10]{ai25}.
\begin{lem}\label{lem_covering}
    Let $R>\rho>0$. Then there exists a set $\alpha \subset B(0,R)$ with $\#\alpha \leq 3^Q \left(\frac{R}{\rho}\right)^Q $ such that
    \[
    B(0,R) \se \bigcup_{a \in \alpha}B(a,\rho).
    \]
   Consequently
    \[
    N(B(0,R),\rho) \leq 3^Q \left(\frac{R}{\rho}\right)^Q.
    \]
\end{lem}
\begin{proof}
    Let $\alpha \subset B(0,R)$ be a maximal set\footnote{The existence and finiteness of such a set are guaranteed by the compactness of closed balls in Carnot groups.} such that $d(a,b) \geq \rho$ for every $a,b \in \alpha$, $a \neq b$. The latter implies 
    \[
    B(0,R) \se \bigcup_{a \in \alpha}B(a,\rho).
    \]
    For every $a,b \in \alpha$, $a \neq b$, we have
    \[
    B(a,\rho/2) \cap B(b,\rho/2)=\emptyset,
    \]
    and, since $a \in B(0,R)$, we also have by the triangle inequality
    \[
    B(a,\rho/2) \se B(0,R+\rho/2).
    \]
    By using left invariance and homogeneity of the Lebesgue measure we obtain
    \[
    (\# \alpha) \leb^n(B(0,1))(\rho/2)^Q= (\# \alpha) \leb^n(B(0,\rho/2)) \leq \leb^n(B(0,R+\rho/2))=\leb^n(B(0,1))(R+\rho/2)^Q,
    \]
    hence, since $R>\rho$,
    \[
    \#\alpha \leq \left( \frac{R+\rho/2}{\rho/2}\right)^Q=(2R/\rho+1)^Q < 3^Q \left( \frac{R}{\rho}\right)^Q.
    \]
    
\end{proof}
Now we prove a sub-Riemannian version of Pierce's Lemma, Theorem \ref{teo_pierce} below. 

\begin{teo}\label{teo_pierce}
    Let $\mu$ be a probability measure on $\G$ with finite $(r+\delta)$-moment. Then there exists a constant $c=c(r,\delta,Q)>0$ such that, for every $N \geq 2$, 
\[
V_{N,r}(\mu) \leq c N^{-\frac{r}{Q}}\left(1+\int_\G d(x,0)^{r+\delta}d\mu(x)\right).
\]
\end{teo}

\begin{proof}
For $\Z \ni k \geq -1$ define
\[
S_k \ceq \begin{cases}
    B(0,1) \quad \text{if } k=-1,\\
    B(0,2^{k+1})\setminus B(0,2^k) \quad \text{if } k \geq 0.
\end{cases}
\]
Let $a=a(Q,r,\delta)$ be a positive constant defined as
\[
a \ceq 3 \left(  2+\frac{2^{Q+1}}{1-2^{-Q\delta/r}}  \right)^\frac{1}{Q}
\]
and define, for every  $N \in \N$ and for every $k \geq -1$,
\[
\rho_k=\rho_k(N) \ceq \begin{cases}
    \min(aN^{-1/Q},1) \quad \text{if } k=-1,\\
    \min (aN^{-1/Q}2^{k(1+\delta/r)},2^{k+1}) \quad \text{if } k \geq 0.
\end{cases}
\]
\begin{center}
    \begin{minipage}{0.98\linewidth}
    \centering
    \begin{tikzpicture}[x=1cm,y=1cm, line cap=round, line join=round, every node/.style={font=\footnotesize}]
    \tikzset{
    shellone/.style={draw=black!70, fill=black!4, line width=0.45pt},
    shelltwo/.style={draw=black!70, fill=black!9, line width=0.45pt},
    shellthree/.style={draw=black!70, fill=black!15, line width=0.45pt},
    coverball/.style={draw=black!80, fill=white, line width=0.35pt},
    point/.style={circle, fill=black, inner sep=1.2pt},
    axis/.style={->, draw=black!55, line width=0.35pt},
    title/.style={anchor=west, font=\small, text=black!85},
    note/.style={text=black!75, font=\scriptsize}
    }
    \begin{scope}[shift={(0,0)}]
    \node[title] at (-1.75,1.9) {(a) Dyadic shells};
    \fill[shellthree, even odd rule] (0,-0.1) circle (1.62) (0,-0.1) circle (1.08);
    \fill[shelltwo, even odd rule] (0,-0.1) circle (1.08) (0,-0.1) circle (0.56);
    \fill[shellone] (0,-0.1) circle (0.56);
    \draw[black!55, line width=0.45pt] (0,-0.1) circle (0.56);
    \draw[black!55, line width=0.45pt] (0,-0.1) circle (1.08);
    \draw[black!55, line width=0.45pt] (0,-0.1) circle (1.62);
    \node[point] at (0,0) {};
    \node[note, anchor=north west] at (0.06,-0.04) {$0$};
    \foreach \x/\y/\r in {0.38/0.2/0.17,-0.78/0.42/0.22,0.72/0.72/0.22,0.22/-0.96/0.22,1.12/-0.7/0.27,-0.28/1.22/0.27}{
        \draw[coverball] (\x,\y) circle (\r);
        \node[point] at (\x,\y) {};
    }
    \draw[axis] (0,-0.1) -- (1.62,-0.1) node[midway, below=2pt, note] {$2^{k+1}$};
    \draw[axis] (0,-0.1) -- (1.08,0.02) node[midway, above=1pt, note] {$2^k$};
    \node[note, fill=white, inner sep=1pt] at (0.15,0.47) {$S_{-1}$};
    \node[note, fill=white, inner sep=1pt] at (0.7,1.0) {$S_k$};
    \node[note, fill=white, inner sep=1pt] at (-0.95,1.25) {$S_{k+1}$};
    \end{scope}
    \begin{scope}[shift={(4.5,0)}]
    \node[title] at (-1.75,1.9) {(b) Scale \(\rho_k\)};
    \draw[axis] (-1.45,-0.55) -- (1.35,-0.55);
    \node[note, anchor=west] at (1.15,-0.34) {$d(x,0)$};
    \foreach \x/\lab in {-0.55/$2^k$, 1.05/$2^{k+1}$}{
        \draw[black!50] (\x,-0.65) -- (\x,-0.45);
        \node[anchor=north, note] at (\x,-0.7) {\lab};
    }
    \fill[black!8] (-0.55,-0.15) rectangle (1.05,0.52);
    \draw[black!75, line width=0.55pt] (-0.55,-0.15) rectangle (1.05,0.52);
    \node[note] at (0.25,0.18) {$S_k$};
    \foreach \x/\y in {-0.28/0.18,0.32/0.32,0.86/0.03}{
        \draw[coverball] (\x,\y) circle (0.22);
        \node[point] at (\x,\y) {};
    }
    \draw[<->, blue!65!black, line width=0.55pt] (-0.28,1.05) -- (0.32,1.05);
    \node[blue!65!black, anchor=south] at (0.02,1.08) {$\rho_k$};
    \draw[blue!65!black, dashed, line width=0.4pt] (-0.28,1.0) -- (-0.28,0.4);
    \draw[blue!65!black, dashed, line width=0.4pt] (0.32,1.0) -- (0.32,0.5);
    \end{scope}
    \end{tikzpicture}

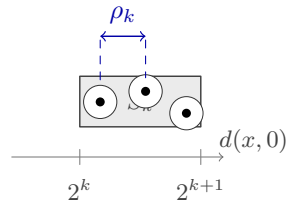
    {\small \textbf{Figure.} Schematic construction in Pierce's lemma: dyadic shells \(S_k\) are covered at scale \(\rho_k\).}
    \end{minipage}
\end{center}
For every $k \geq - 1$ we define $\alpha_k \subset B(0,2^{k+1})$ as follows:
\begin{itemize}
    \item[(i)] If $\rho_k<2^{k+1}$, by Lemma \ref{lem_covering} we can find a set $\alpha_k$ such that
    \[
    \#\alpha_k \leq 3^Q \left( \frac{2^{k+1}}{\rho_k}  \right)^Q, \quad \min_{a \in \alpha_k} d(x,a) <\rho_k \quad\text{ for every } x \in B(0,2^{k+1}).
    \]
    \item[(ii)] If $\rho_k=2^{k+1}$ we set $\alpha_k \ceq \{0\}$ and we observe that $d(x,0)\leq 2^{k+1}=\rho_k$ for every $x \in B(0,2^{k+1})$.
\end{itemize}
Finally we define
\[
\alpha \ceq \bigcup_{k \geq -1}\alpha_k,
\]
and we observe that for every $x \in S_k \se B(0,2^{k+1})$ we have
\[
\min_{a \in \alpha}d(x,a) \leq \min_{a \in \alpha_k} d(x,a) \leq \rho_k.
\]
For all sufficiently large $k$, one has $\rho_k=2^{k+1}$, hence $\alpha_k=\{0\}$; in particular, $\alpha$ is finite.
Moreover,
\[
\#\alpha  \leq 1+ 3^Q \left(\frac{1}{aN^{-1/Q}}\right)^Q +\sum_{k \geq 0} 3^Q \left( \frac{2^{k+1}}{aN^{-1/Q}2^{k(1+\delta/r)}}   \right)^Q \leq 1+N3^Qa^{-Q}\left(1+\frac{2^Q}{1-2^{-Q\delta/r}}\right).
\]
Thanks to our choice of $a$ we obtain
\[
\#\alpha \leq 1+\frac{N}{2},
\]
hence $\#\alpha \leq N$ provided that $N\geq 2$. Therefore we obtain
\begin{equation}\label{eq_aaaaa}
    V_{N,r}(\mu)\leq \int_\G  \min_{a \in \alpha}d(x,a)^rd\mu(x) \leq \sum_{k \geq -1}\rho_k^r \mu(S_k).
\end{equation}
By the definition of $\rho_k$ we have
\[
\rho^r_{-1} \leq a^r N^{-\frac{r}{Q}}
\]
and, for $k \geq 0$,
\[
\rho^r_k \leq a^r N^{-\frac{r}{Q}}2^{k(r+\delta)}.
\]
Combining the latter with \eqref{eq_aaaaa} we obtain
\begin{equation}\label{eq_bbbb}
V_{N,r}(\mu) \leq a^r N^{-\frac{r}{Q}}\left(\mu(S_{-1})+\sum_{k \geq 0}2^{k(r+\delta)}\mu(S_k)\right).  
\end{equation}
We recall that for every $x \in S_k$ we have $d(x,0) \geq 2^k$, hence by elevating to the power $r+\delta$ and integrating on $S_k$ we obtain
\[
2^{k(r+\delta)}\mu(S_k) \leq \int_{S_k}d(x,0)^{r+\delta}d\mu(x).
\]
Combining the latter with \eqref{eq_bbbb} we obtain
\[
V_{N,r}(\mu) \leq c N^{-r/Q}\left(1+\int_\G d(x,0)^{r+\delta}d\mu(x)\right),
\]
where $c\ceq a^r$. This concludes the proof.
\end{proof}

We observe in passing that one can reason in a similar way as in the construction of $\alpha$ in the proof of Theorem \ref{teo_pierce} above to obtain the following result.
\begin{cor}\label{cor_net}
Let $K\subset \G$ be a non-empty compact set. Then there exists a constant $A=A(K)>0$ such that
for every $N\in\N$ there exists a finite set
\[
\Lambda_N\subset K
\]
satisfying
\[
\#\Lambda_N\leq N
\qquad \text{and} \qquad
\sup_{x\in K} d(x,\Lambda_N)\le AN^{-\frac{1}{Q}}.
\]
\end{cor}
\begin{proof}
    Since $K$ is compact, there exists $R>0$ such that $K \subset B(0,R)$. By Lemma \ref{lem_covering}, for every $R>\rho>0$ we have
    \[
 N(K,\rho) \leq N(B(0,R),\rho) \leq (3R)^Q\rho^{-Q},
    \]
    hence we can find a set $\tilde\Gamma_\rho \se \G$ such that
    \[
    \#\tilde\Gamma_\rho \leq (3R)^Q \rho^{-Q} \qquad \text{and} \qquad \sup_{x \in K}d(x,\tilde\Gamma_\rho) \leq \rho.
    \]
  Discarding, if necessary, all centers $\tilde x\in\tilde\Gamma_\rho$ such that $K\cap B(\tilde x,\rho)=\emptyset$, the remaining family still covers $K$ and has no larger cardinality. For every remaining $\tilde x \in \tilde \Gamma_\rho$ we choose $x(\tilde x) \in K \cap B(\tilde x,\rho)$ and we define $\Gamma_\rho \ceq \{x(\tilde x): \tilde x \in \tilde\Gamma_\rho\}$. We have that $\Gamma_\rho \se K$,
     \[
    \#\Gamma_\rho \leq (3R)^Q \rho^{-Q} \qquad \text{and} \qquad \sup_{x \in K}d(x,\Gamma_\rho) \leq 2\rho.
    \]
    For every $N \in \N$ we define 
    \[
    \rho_N \ceq 3RN^{-\frac{1}{Q}}
    \]
    If $\rho_N>R$, then choosing any point $\lambda_0 \in K$ and setting $\Lambda_N \ceq \{ \lambda_0\}$ gives
\begin{equation}\label{eq_case1}
    \#\Lambda_N =1 \leq N \qquad\text{and}\qquad \sup_{x \in K}d(x,\Lambda_N) \leq \operatorname{diam}(K) \leq 2R<2\rho_N=6RN^{-\frac{1}{Q}}.
    \end{equation}
    If $\rho_N \leq R$ we set $\Lambda_N \ceq \Gamma_{\rho_N}$. Then   \begin{equation}\label{eq_case2}
 \#\Lambda_N=\#\Gamma_{\rho_N} \leq (3R)^Q \rho^{-Q}_N \leq N \qquad \text{ and }\qquad \sup_{x \in K}d(x,\Lambda_N) =\sup_{x \in K}d(x,\Gamma_{\rho_N})\leq 2\rho_N=6RN^{-\frac{1}{Q}}      
    \end{equation}
Defining $A(K) \ceq 6R $ and combining \eqref{eq_case1} and \eqref{eq_case2} is enough to conclude.
\end{proof}

\section{Asymptotics of quantization of measures in Carnot groups}
\label{sec_qom}

The aim of this section is to prove the two main results stated in the introduction: for the reader's convenience, we outline here the structure of the rest of this section.
\begin{enumerate}
    \item In Lemma \ref{lem_step1} we prove Theorem \ref{teo_intro} in the base case, that is, when $\mu$ is the uniform measure on the Euclidean cube $C$. We refer the reader to the beginning of Subsection \ref{subsec_zad1} for further details on the ideas behind the strategy of the proof.
    \item In Lemma \ref{lem_step2} we prove Theorem \ref{teo_intro} in the specific case when $\mu$ is a convex linear combination of  uniform measures supported on disjoint copies of left translations and/or homogeneous dilations of the Euclidean cube $C$.
    \item In Lemma \ref{lem_step3} (resp. Lemma \ref{lem_step4}) we prove Theorem \ref{teo_intro} when $\mu$ is a compactly supported measure absolutely continuous (resp. mutually singular) with respect to $\leb^n$. Lemma \ref{lem_step3} is proved by approximating $\mu$ with a sequence of probability measures $\mu_k$, each of them being supported on copies of left translations and/or homogeneous dilations of the Euclidean cube $C$, allowing us to use Lemma \ref{lem_step2}. Lemma \ref{lem_step4} is proved by choosing an appropriate tessellation of the support of $\mu$.
    \item Lemma \ref{lem_step3} and Lemma \ref{lem_step4} allow us to prove Theorem \ref{teo_intro} when $\mu$ is a compactly supported probability measure. If $\mu$ is not compactly supported, then it is always true (Lemma \ref{lem_stepaux}) that
    \[
    \liminf_{N \to +\infty}N^{\frac{r}{Q}}V_{N,r}(\mu) \geq Q_r(C) \nl h \nr_{L^{\frac{Q}{Q+r}}},
    \]
    but, in order to obtain the analogous inequality for the $\limsup$, we need a stronger assumption on the finiteness of moments which allows us to use our sub-Riemannian version of Pierce's Lemma (Theorem \ref{teo_intropierce}).
    Finally we prove Theorem \ref{teo_intro} for every probability measure $\mu$ with finite $(r+\delta)$-moment by taking a sequence of increasingly homogeneous dilations of the Euclidean cube $C$, $(\delta_\lambda C)_{\lambda \in \N}$, and by decomposing $\mu$ as the sum of a measure supported on the cube and on the complement of the cube. By Corollary \ref{cor_compact_zador} we know how to estimate the quantization error on the compactly supported measure and by Theorem \ref{teo_intropierce} we can obtain estimates for the quantization error on the non-compactly supported measure.
    \item We conclude the section by proving Theorem \ref{teo_empirical_centers}, using the localized lower-bound argument and the equality case in the allocation inequality.
 \end{enumerate}

\subsection{Zador's Theorem for the Euclidean cube in Carnot groups} \label{subsec_zad1}

The aim of this subsection is to prove Theorem \ref{teo_intro} for the probability measure $U(C)$. Our strategy is the following. First, in Lemma \ref{lem_pav} we approximate the Euclidean cube $C$ by using a sequence of sets $(C_\lambda)_{\lambda>0}$ that are finite disjoint unions of suitably left-translated and homogeneously dilated copies of $C$. Then, in Lemma \ref{lem_step1} we prove Zador's Theorem for the Euclidean cube by estimating the quantization error of $U(C)$ with the ones of $U(C_\lambda)$ and $U(C \setminus C_\lambda)$. Moreover, thanks to the fact that moments are minimized by balls (Lemma \ref{lem_momenti}), we are able to prove in Lemma \ref{lem_qpositivo} that the asymptotic constant $Q_r(C)$ is strictly positive.
\begin{lem}\label{lem_pav}
For every $\lambda>0$ define
\[
P_\lambda \ceq  \{ \delta_\lambda(p \cdot C): p \in \Z^{n} \} ,\qquad T_\lambda\ceq \{E \in P_\lambda: E \subset \operatorname{int}C\},
\]
and 
\[
C_\lambda \ceq \bigsqcup_{A \in T_\lambda}A.
\]
Then $\leb^n(C_\lambda)=\lambda^Q(\#T_\lambda)$ and
\[
\lim_{\lambda \to 0}\leb^n(C \setminus C_\lambda)=0.
\]
\end{lem}
\begin{proof}
Let $\lambda>0$. The fact that $\leb^n(C_\lambda)=\lambda^Q(\#T_\lambda)$ immediately follows by the fact that $C_\lambda$ is a disjoint union and from the left-invariance and homogeneity of the Lebesgue measure. Let $D \ceq \operatorname{diam}(C)<+\infty$. By left-invariance and homogeneity of the distance, we also have $\operatorname{diam}( \delta_\lambda(p \cdot C))\leq \lambda D$. Let $x \in \operatorname{int(C)}$: then $x \in \delta_\lambda(\bar p \cdot C)$ for some\footnote{The existence and the uniqueness of such $\bar p$ is guaranteed by Lemma \ref{lem_pavage}.} $\bar p \in \Z^n$. If $\delta_\lambda(\bar p \cdot C) \not\subset \operatorname{int}(C)$ then there exists $y \in \delta_\lambda(\bar p \cdot C) \setminus \operatorname{int}(C)$ and
\[
d(x,\G \setminus \operatorname{int}(C)) \leq d(x,y) \leq \lambda D.
\]
Consequently,
\[
C \setminus C_\lambda \se (C \setminus \operatorname{int}(C))\cup \{x \in \operatorname{int}(C):d(x,\G \setminus \operatorname{int}(C)) \leq \lambda D\}.
\]
The set $C \setminus \operatorname{int}(C)$ has zero Lebesgue measure and the set $\{x \in \operatorname{int}(C):d(x,\G \setminus \operatorname{int}(C)) \leq \lambda D\}$ becomes the empty set when $\lambda \to 0$, more precisely, these sets decrease to the empty set as $\lambda\to 0$, so one may use continuity from above, proving that 
\[
\lim_{\lambda \to 0}\leb^n(C \setminus C_\lambda)=0,
\]
concluding the proof.
\end{proof}
\begin{lem} \label{lem_step1}
    We have that
    \[
    \lim_{N \to +\infty} N^\frac{r}{Q}V_{N,r}(U(C))=Q_r(C).
    \]
\end{lem}
\begin{proof}
  Fix $\ve>0$ and choose $S \in \N$, $0<\sigma<1/2$ such that 
  \[
  S^\frac{r}{Q}V_{S,r}(U(C))<Q_r(C)+\ve, \qquad (1-2\sigma)^{-\frac{r}{Q}}\leq 1+\ve
  \]
 Define, for every $N \in \N$,
  \[
  N_1 \ceq [(1-\sigma)N], \qquad N_2 \ceq N-N_1, \qquad \lambda_N \ceq \left( \frac{S}{N_1} \right)^\frac{1}{Q},
  \]
  where, here and in the following, $[\cdot]$ denotes the integer part function. We observe in passing that $\lambda_N \xrightarrow{N \to +\infty}0$. Let 
  \[
  m_N \ceq \leb^n(C_{\lambda_N}),
  \]
  where $C_{\lambda_N}$ is defined as in Lemma \ref{lem_pav}. By Lemma \ref{lem_pav} we have
  \[
  1 \geq m_N=(\lambda_N)^Q(\#T_{\lambda_N}) \xrightarrow{N \to +\infty}\leb^n(C)=1,
  \]
  where we used the same notation of Lemma \ref{lem_pav}. Therefore, by the definition of $\lambda_N$,
  \[
  S(\#T_{\lambda_N})=S\frac{m_N}{(\lambda_N)^Q}=m_N N_1 \leq N_1.
  \]
  We have
  \[
  U(C)=m_NU(C_{\lambda_N})+(1-m_N)U(C \setminus C_{\lambda_N}).
  \]
  We now choose $S(\#T_{\lambda_N})$ centers on $C_{\lambda_N}$ and $N_2$ centers on $C \setminus C_{\lambda_N}$: by Lemma \ref{lem_basicpro} we infer
  \begin{equation}\label{eq_argh}
  V_{N,r}(U(C)) \leq m_NV_{ S(\#T_{\lambda_N}),r}(U(C_{\lambda_N}))+(1-m_N)V_{N_2,r}(U(C \setminus C_{\lambda_N})).
  \end{equation}
  Since 
  \[
  U(C_{\lambda_N})=\frac{1}{\#T_{\lambda_N}}\sum_{i=1}^{\#T_{\lambda_N}}U(\delta_{\lambda_N}(p_i \cdot C))
  \]
  for some $p_1,\dots,p_{\#T_{\lambda_N}} \in \Z^n$, we obtain by Lemma \ref{lem_basicpro} and by recalling the definition of $\lambda_N$,
  \[
V_{ S(\#T_{\lambda_N}),r}(U(C_{\lambda_N})) \leq \frac{1}{\#T_{\lambda_N}}\sum_{i=1}^{\#T_{\lambda_N}}V_{S,r}(U(\delta_{\lambda_N}(p_i \cdot C)))=\lambda_N^rV_{S,r}(U(C))=\left(\frac{S}{N_1}\right)^\frac{r}{Q}V_{S,r}(U(C)).
  \]
  Combining the latter with \eqref{eq_argh} we obtain, by multiplying both sides of the inequality by $N^\frac{r}{Q}$,
  \begin{equation}\label{eq_argh2}
  N^\frac{r}{Q}V_{N,r}(U(C)) \leq \left(\frac{N}{N_1}\right)^\frac{r}{Q}S^\frac{r}{Q} m_N V_{S,r}(U(C))+N^\frac{r}{Q}(1-m_N)V_{N_2,r}(U(C \setminus C_{\lambda_N})).
  \end{equation}
  We also observe that 
  \begin{align}\label{eq_argh3}
  N^\frac{r}{Q}(1-m_N)V_{N_2,r}(U(C \setminus C_{\lambda_N}))=&\left(\frac{N}{N_2}\right)^\frac{r}{Q} (1-m_N)N_2^\frac{r}{Q}V_{N_2,r}(U(C \setminus C_{\lambda_N})) \\\leq& \sigma^{-\frac{r}{Q}}(1-m_N)N_2^\frac{r}{Q}V_{N_2,r}(U(C \setminus C_{\lambda_N})).\notag
  \end{align}
  Moreover, since $U(C \setminus C_{\lambda_N})$ is supported on the fixed compact set $C$, Pierce's Lemma (Theorem \ref{teo_pierce}), applied with $\delta=1$, gives a constant $M>0$, depending only on $r,Q$, the chosen homogeneous distance, and the cube $C$, such that 
  \[
  N_2^\frac{r}{Q}V_{N_2,r}(U(C \setminus C_{\lambda_N})) \leq M
  \]
  Combining the above inequality with \eqref{eq_argh2} and \eqref{eq_argh3} and observing that, whenever $N \geq 1/\sigma$ we have $N/N_1 \leq (1-2\sigma)^{-1}$, we obtain
  \[
  N^\frac{r}{Q}V_{N,r}(U(C)) \leq (1-2\sigma)^{-\frac{r}{Q}}S^\frac{r}{Q} m_N V_{S,r}(U(C))+\sigma^{-\frac{r}{Q}}(1-m_N)M.
  \]
  Since $m_N \xrightarrow{N \to +\infty}1$ we have
  \[
  \lim_{N \to +\infty}\sigma^{-\frac{r}{Q}}(1-m_N)M=0,
  \]
 and, consequently,
  \[
  \limsup_{N \to +\infty} N^\frac{r}{Q}V_{N,r}(U(C)) \leq  (1-2\sigma)^{-\frac{r}{Q}}S^\frac{r}{Q}  V_{S,r}(U(C)).
  \]
  Recalling the definition of $S$ and $\sigma$ we obtain 
    \[
  \limsup_{N \to +\infty} N^\frac{r}{Q}V_{N,r}(U(C)) \leq  (1+\ve)(Q_r(C)+\ve),
  \]
  and letting $\ve \to 0$ gives the desired upper bound. The reverse inequality follows directly from the definition of $Q_r(C)$, since $N^{r/Q}V_{N,r}(U(C))\geq Q_r(C)$ for every $N$.
\end{proof}

Before proving that $  Q_r(C)>0$, we need an auxiliary lemma which proves that moments are minimized by balls.
\begin{lem}\label{lem_momenti}
Let $V \se \G$ be a bounded Lebesgue measurable set such that $\leb^n(V)>0$. Then for every $x_0 \in \G$ there exists $R \geq 0$ such that $\leb^n(B(x_0,R))=\leb^n(V)$ and
\[
\int_V d(x,x_0)^r dx \geq \int_{B(x_0,R)} d(x,x_0)^r dx.
\]
\end{lem}
\begin{proof}
    Since $V$ is bounded, there exists $R_0>0$ such that $V \se B(x_0,R_0)$, with measure $M \geq \leb^n(V)$. We define the function $\phi:[0,R_0]\to [0,M]$ as $\phi(R) \ceq \leb^n(B(x_0,R))$. The function $\phi$ is monotone and left-continuous since $B(x_0,R)=\bigcup_{S \to R^-}B(x_0,S)$. Moreover, since $\bigcap_{S>R}B(x_0,S)=\overline{B(x_0,R)}$ and (see \cite[Proposition 2.4]{fssc03})  $\leb^n(\partial B(x_0,R))=0$ and $\leb^n(B(x_0,R))=\lim_{S \to R^+} \leb^n(B(x_0,S))$, $\phi$ is also right-continuous. Therefore, by the intermediate value theorem there exists $R \in [0,R_0]$ such that $\leb^n(B(x_0,R))=\leb^n(V)$. Equivalently, $\leb^n(V \setminus B(x_0,R))=\leb^n(B(x_0,R) \setminus V)$, and since $d(\cdot,x_0)\geq R$ on the former set and $d(\cdot,x_0) \leq R$ on the latter, one obtains
\begin{align*}
    \int_Vd(y,x_0)^rdy -\int_{B(x_0,R)} d(y,x_0)^rdy &= \int_{V \setminus B(x_0,R)}d(y,x_0)^rdy-\int_{ B(x_0,R)  \setminus V}d(y,x_0)^rdy\\& \geq \int_{V \setminus B(x_0,R)}R^rdy-\int_{ B(x_0,R)  \setminus V}R^rdy\\
    &=R^r(\leb^n(V \setminus B(x_0,R))-\leb^n(B(x_0,R) \setminus V))=0,
\end{align*}
concluding the proof.
\end{proof}

\begin{lem}\label{lem_qpositivo}
    We have
    \[
    Q_r(C)>0.
    \]
\end{lem}
\begin{proof}
For every $N \in \N$ choose $\alpha_N\in C_{N,r}(U(C))$. Since $\operatorname{spt}(U(C))$ is infinite, Lemma \ref{lem_card} (and the trivial case $N=1$) gives $\#\alpha_N=N$; we write $\alpha_N=\{a_1,\dots,a_N\}$. Consider, for $1 \leq i \leq N$,
\[
\tilde V_i \ceq \{p \in \G: d(p,a_i) \leq d(p,a_j), \text{ for all } j \neq i\},
\] 
and
\[
 V_1 \ceq \tilde V_1, \qquad V_i \ceq  \tilde V_i \setminus  \left(  \bigcup_{j<i} \tilde V_j  \right) \text{ if }i \geq 2.
\]
Then we have
\begin{equation}\label{eq_v1}
V_{N,r}(U(C)) = \sum_{i=1}^N \int_{V_i \cap C}d(x,a_i)^r dx.
\end{equation}
For every $1 \leq i \leq N$ we apply Lemma \ref{lem_momenti} with $V=V_i \cap C$ and $x_0=a_i$, hence we can find $R_i\geq 0$ such that $ \leb^n(B(a_i,R_i))=\leb^n(V_i \cap C) \eqqcolon m_i$ and
\begin{equation}\label{eq_vaux}
\int_{V_i \cap C}d(x,a_i)^r dx  \geq \int_{B(a_i,R_i)}d(x,a_i)^rdx.
\end{equation}
By left-invariance and homogeneity we obtain 
\[
R_i= \left(\frac{m_i}{\leb^n (B(0,1))}  \right)^\frac{1}{Q},
\]
and
\begin{equation}\label{eq_v2}
\int_{B(a_i,R_i)}d(x,a_i)^rdx=R_i^{Q+r}\int_{B(0,1)}d(x,0)^rdx= \left(\frac{m_i}{\leb^n (B(0,1))}  \right)^\frac{Q+r}{Q}\int_{B(0,1)}d(x,0)^rdx.
\end{equation}
We observe
\begin{align}\label{eq_v3}
\int_{B(0,1)}d(x,0)^rdx&=\int_0^1 \leb^n(\{x \in B(0,1):d(x,0)^r \geq t \})dt
   \\&=\int_0^1 rs^{r-1}\leb^n(\{x \in B(0,1):d(x,0) \geq s \})ds \notag
   \\&=\int_0^1 rs^{r-1}\big(\leb^n(B(0,1))-\leb^n(B(0,s))\big)ds \notag
    \\&=\leb^n(B(0,1))\int_0^1 rs^{r-1}\big(1-s^Q\big)ds \notag
    \\&=\leb^n(B(0,1))\frac{Q}{Q+r}.\notag
\end{align}
By \eqref{eq_v1}, \eqref{eq_vaux}, \eqref{eq_v2} and \eqref{eq_v3} we obtain
\begin{equation}\label{eq_vv1}
V_{N,r}(U(C)) \geq \leb^n (B(0,1))^{-\frac{r}{Q}}\frac{Q}{Q+r} \sum_{i=1}^N (m_i)^\frac{Q+r}{Q}.
\end{equation}
Since $\tfrac{Q+r}{Q}>1$ the map $t \to t^{\tfrac{Q+r}{Q}}$ is convex; moreover $\sum_{i=1}^N m_i=\leb^n(C)=1$ hence, by Jensen's inequality, we obtain 
\begin{equation}\label{eq_vv2}
\sum_{i=1}^N (m_i)^\frac{Q+r}{Q} \geq N^{-\frac{r}{Q}}.
\end{equation}
From \eqref{eq_vv1} and \eqref{eq_vv2} we obtain
\[
V_{N,r}(U(C)) \geq \leb^n (B(0,1))^{-\frac{r}{Q}}\frac{Q}{Q+r} N^{-\frac{r}{Q}}.
\]
By multiplying the latter by $N^\frac{r}{Q}$ and taking the infimum over $N \geq 1$ we obtain
\[
Q_r(C) \geq \leb^n (B(0,1))^{-\frac{r}{Q}}\frac{Q}{Q+r} >0,
\]
concluding the proof.
\end{proof}
\subsection{Zador's Theorem for compactly supported absolutely continuous probability measures}
\label{subsec_zad2}
In Lemma \ref{lem_step2} below we prove Zador's Theorem for measures which are linear combinations of (suitably rescaled) uniform measures supported on disjoint copies $(C_1,\dots,C_m)$ of left translations and/or dilations of the Euclidean cube $C$. For the reader's convenience, we sketch here the main ideas of the proof. We want to prove that the right asymptotic rate for $V_{N,r}(\mu)$ is given by $N^{-\frac{r}{Q}}$ and that the asymptotic constant is obtained by \emph{allocation}. The idea is that, for $N$ large enough, an optimal quantization is obtained as follows:
\begin{enumerate}
    \item we decide how many points $N_i$ to allocate to each $C_i$;
    \item we optimize separately on each $C_i$;
    \item we sum the errors.
\end{enumerate}
In order to achieve this, we need two estimates: a lower bound, namely that every (global) choice of quantizers has to pay at least the sum of the local costs, and a (constructive) upper bound, namely that we construct a quantizer by putting $N_i$ points in every $C_i$ and obtain the desired error. In other words, the global problem is reduced to a problem of discrete optimization on the ratios $\frac{N_i}{N}$.

\begin{lem}\label{lem_step2}
    Let $m \in \N, \lambda>0$ and $s_i \geq 0$ such that $\sum_{i=1}^ms_i=1$. Let $\mu \ceq \sum_{i=1}^m s_iU( C_i)$ where each $C_i \se \G$ is obtained as $\delta_{\lambda} (p_i \cdot C)$ for some $p_i \in \Z^{n}$  and $\leb^{n}(C_i \cap C_j)=0$ for every $i \neq j$. Then
\[
\lim_{N \to +\infty}N^{\frac{r}{Q}}V_{N,r}(\mu)=Q_r(C)\nl \frac{d\mu}{d\leb^{n}} \nr_{L^{\frac{Q}{Q+r}}}
\]
where $\frac{d\mu}{d\leb^{n}}$ is the density of $\mu$ with respect to $\leb^{n}$, i.e., $\frac{d\mu}{d\leb^{n}}=\sum_{i=1}^m s_i \lambda^{-Q}\uno_{C_i}$.
\end{lem}
\begin{proof}
   Without loss of generality, we assume $s_i>0$ for every $i$. For $i \in \{1,\dots,m\}$  we define
    \[
    t_i \ceq \frac{s_i^\frac{Q}{Q+r}}{\sum_{j=1}^m s_j^\frac{Q}{Q+r}},
    \]
    and for $N \geq \max_{1 \leq i \leq m} 1/t_i$ we define
    \[
     N_i=N_i(N) \ceq [t_iN],
    \]
Since $\sum_{i=1}^m N_i\leq N$, by Lemma \ref{lem_basicpro},
    \[
    V_{N,r}(\mu) \leq \sum_{i=1}^m s_iV_{N_i,r}(U(C_i))=\lambda^{r}\sum_{i=1}^m s_i V_{N_i,r}(U(C)).
    \]
By Lemma \ref{lem_step1}
\[
N^{\frac{r}{Q}}V_{N_i,r}(U(C))=\left(\frac{N}{N_i}\right)^\frac{r}{Q}N_i^{\frac{r}{Q}}V_{N_i,r}(U(C))\xrightarrow{N \to +\infty}t_i^{-\frac{r}{Q}} Q_r(C).
\]
The latter implies 
\begin{equation}\label{eq_limsup}
    \limsup_{N \to +\infty} N^{\frac{r}{Q}}V_{N,r}(\mu) \leq Q_r(C) \sum_{i=1}^m s_i t_i^{- \frac{r }{Q}}\lambda^{r}=Q_r(C) \nl \frac{d\mu}{d\leb^{n}} \nr_{L^{\frac{Q}{Q+r}}}.
\end{equation}
 We are left to prove that 
 \[   
    \liminf_{N \to +\infty}N^{\frac{r}{Q}}V_{N,r}(\mu)
    \geq Q_r(C)\nl \frac{d\mu}{d\leb^{n}} \nr_{L^{\frac{Q}{Q+r}}}.
    \]
Let $0<\rho<1$ and define, for every $i=1,\dots,m$,
\[
C^\rho \ceq \delta_\rho(C) \subset C,\qquad C^\rho_i \ceq \delta_\lambda(p_i 
\cdot C^\rho).
\]
Since $C^\rho\subset C$, for every $i$ the set $C_i^\rho$ is compactly contained in
$\operatorname{int}(C_i)$. Therefore the function
\[
x\longmapsto d(x,\G\setminus \operatorname{int}(C_i))
\]
is continuous and strictly positive on $C_i^\rho$. Hence there exists a finite set
$\Gamma_i^\rho\subset C_i^\rho$ such that
\[
d(x,\Gamma_i^\rho)\le d(x,\G\setminus \operatorname{int}(C_i))
\qquad\text{for every }x\in C_i^\rho.
\]
Since all the sets $C_i^\rho$ are left-translated/dilated copies of the same set,
we may assume
\[
\#\Gamma_i^\rho=k_\rho
\]
for every $i$. Let $\alpha_N\in C_{N,r}(\mu)$ and define
\[
N_i(N)\ceq \#\bigl(\alpha_N\cap \operatorname{int}(C_i)\bigr).
\]
Passing to a subsequence if necessary, we may assume that
\[
\frac{N_i(N)}{N}\xrightarrow[N\to+\infty]{} v_i\in[0,1],
\qquad i=1,\dots,m.
\]
Since the interiors of the $C_i$ are pairwise disjoint, we have
\[
\sum_{i=1}^m v_i\le 1.
\]
Fix $i\in\{1,\dots,m\}$ and $x\in C_i^\rho$. If $a\in \alpha_N\setminus \operatorname{int}(C_i)$, then
\[
d(x,a)\ge d(x,\G\setminus \operatorname{int}(C_i))\ge d(x,\Gamma_i^\rho),
\]
therefore
\[
d(x,\alpha_N\cup \Gamma_i^\rho)
=
d\bigl(x,(\alpha_N\cap \operatorname{int}(C_i))\cup \Gamma_i^\rho\bigr).
\]
Using this identity, we obtain
\begin{align*}
V_{N,r}(\mu)
&=
\sum_{i=1}^m s_i\int_{C_i} d(x,\alpha_N)^rdU(C_i)(x)\\
&\ge
\sum_{i=1}^m s_i\int_{C_i^\rho} d(x,\alpha_N)^rdU(C_i)(x)\\
&\ge
\sum_{i=1}^m s_i\int_{C_i^\rho} d\bigl(x,(\alpha_N\cap \operatorname{int}(C_i))\cup \Gamma_i^\rho\bigr)^rdU(C_i)(x).
\end{align*}
Since
\[
\#\bigl((\alpha_N\cap \operatorname{int}(C_i))\cup \Gamma_i^\rho\bigr)\le N_i(N)+k_\rho,
\]
it follows from the definition of $V_{N,r}$ that
\[
\int_{C_i^\rho} d\bigl(x,(\alpha_N\cap \operatorname{int}(C_i))\cup \Gamma_i^\rho\bigr)^r\,dU(C_i)(x)
\geq
\frac{\leb^n(C_i^\rho)}{\leb^n(C_i)}\,V_{N_i(N)+k_\rho,r}(U(C_i^\rho)).
\]
Therefore
\[
V_{N,r}(\mu)\geq
\sum_{i=1}^m s_i\frac{\leb^n(C_i^\rho)}{\leb^n(C_i)}
V_{N_i(N)+k_\rho,r}(U(C_i^\rho)).
\]
We observe that
\[
\frac{\leb^n(C_i^\rho)}{\leb^n(C_i)}=\rho^Q,
\]
and, by left invariance and homogeneity,
\[
V_{N_i(N)+k_\rho,r}(U(C_i^\rho))=(\lambda\rho)^rV_{N_i(N)+k_\rho,r}(U(C)).
\]
We claim that $v_i>0$ for every $i$. Indeed, if $v_i=0$ for some $i$, then the above lower bound, Lemma \ref{lem_qpositivo}, and the definition of $Q_r(C)$ give
\begin{align*}
N^{\frac rQ}V_{N,r}(\mu)
&\geq s_i\rho^{Q+r}\lambda^r
\left(\frac{N}{N_i(N)+k_\rho}\right)^{\frac rQ}
(N_i(N)+k_\rho)^{\frac rQ}V_{N_i(N)+k_\rho,r}(U(C))\\
&\geq s_i\rho^{Q+r}\lambda^r Q_r(C)
\left(\frac{N}{N_i(N)+k_\rho}\right)^{\frac rQ}.
\end{align*}
The right-hand side tends to $+\infty$, contradicting the finite upper bound in \eqref{eq_limsup}. Using Lemma \ref{lem_step1}, since $v_i>0$ we infer
\[
\liminf_{N\to+\infty}
N^{\frac rQ}V_{N_i(N)+k_\rho,r}(U(C_i^\rho))
\ge
v_i^{-\frac{r}{Q}}(\lambda\rho)^rQ_r(C).
\]
Consequently
\[
\liminf_{N\to+\infty}N^{\frac rQ}V_{N,r}(\mu)
\geq
\rho^{Q+r}\lambda^rQ_r(C)\sum_{i=1}^m s_i v_i^{-\frac{r}{Q}}.
\]
 By letting $\rho \to 1$ and by \cite[Lemma 6.8]{gl00} we obtain
\[
\liminf_{N\to+\infty}N^{\frac rQ}V_{N,r}(\mu)
 \geq Q_r(C)\lambda^{r}\left( \sum_{i=1}^m s_i^\frac{Q}{Q+r} \right)^\frac{Q+r  }{Q}= Q_r(C) \nl \frac{d\mu}{d\leb^{n}} \nr_{L^{\frac{Q}{Q+r}}},
\]
concluding the proof.
\end{proof}

Before proving Zador's Theorem for compactly supported absolutely continuous measures we need an auxiliary technical lemma.

\begin{lem}\label{lem_convhk}
    Let $\mu$ be absolutely continuous with respect to $\leb^n$ and with compact support, i.e., $\mu=h\leb^n$ and $\operatorname{spt}(h) \subset K$ for some compact set $K \subset \G$. For every $k \in \N$ let $\mathcal T_k$ be defined as
    \[
 \mathcal T_k \ceq \left\lbrace   \delta_{1/k}(\gamma \cdot C) : \gamma \in \Z^n \right\rbrace,
    \]
    and $\mathcal P_k$ as the finite subset of all the elements of $\mathcal T_k$ that have non-empty intersection with $K$. Define, for every $k \in \N$
    \[
    h_k \ceq \sum_{E \in \mathcal P_k} \frac{\mu(E)}{\leb^n(E)} \uno_{E},
    \]
    i.e., $h_k$ is the density (with respect to $\leb^n$) of the measure 
    \[
    \mu_k \ceq \sum_{E \in \mathcal P_k} \mu(E)U(E).
    \]
Then 
\[
\nl h_k-h \nr_{L^1} \xrightarrow{k \to +\infty}0.
\]
Consequently
\[
\|h_k-h\|_{L^{\frac{Q}{Q+r}}}\xrightarrow{k\to+\infty}0
\qquad\text{and}\qquad
\|h_k\|_{L^{\frac{Q}{Q+r}}}\xrightarrow{k\to+\infty}\|h\|_{L^{\frac{Q}{Q+r}}}.
\]
\end{lem}
\begin{proof}
Let $k \in \N$, $x \in K$ and $\bar p \in \Z^n$ be such that\footnote{The existence and the uniqueness of such $\bar p$ is guaranteed by Lemma \ref{lem_pavage}.} $x \in \delta_{1/k}(\bar p \cdot C) \in \mathcal P_k$. Then
\[
h_k(x)=\frac{\mu(\delta_{1/k}(\bar p \cdot C) )}{\leb^n (\delta_{1/k}(\bar p \cdot C) )}=\frac{1}{\leb^n (\delta_{1/k}(\bar p \cdot C) )}\int_{\delta_{1/k}(\bar p \cdot C) }h(y)dy.
\]
If we denote by $D \ceq \operatorname{diam}C$ then we have that 
\[
\delta_{1/k}(\bar p \cdot C) \subset B(x,D/k) 
\]
and we can write
\begin{align*}
|h_k(x)-h(x)|&\leq \frac{1}{\leb^n (\delta_{1/k}(\bar p \cdot C) )}\int_{\delta_{1/k}(\bar p \cdot C) }|h(y)-h(x)|dy \\&\leq \frac{\leb^n(B(x,D/k))}{\leb^n (\delta_{1/k}(\bar p \cdot C) )}\frac{1}{\leb^n(B(x,D/k))}\int_{B(x,D/k)}|h(y)-h(x)|dy
\end{align*}
Since
\[
\frac{\leb^n(B(x,D/k))}{\leb^n (\delta_{1/k}(\bar p \cdot C) )}=\frac{\leb^n(B(0,1))k^{-Q}D^{Q}}{\leb^n(C)k^{-Q}}=\leb^n(B(0,1))D^{Q}
\]
is uniformly bounded, by Lebesgue differentiation theorem,
    \begin{equation}\label{eq_rrr}
    h_k \xrightarrow{k \to +\infty} h  \qquad  \leb^{n}\text{-a.e. on }K.
    \end{equation}
On the other hand, if $x \not \in K$ let $\bar p \in \Z^n$ be such that\footnote{The existence and the uniqueness of such $\bar p$ is guaranteed by Lemma \ref{lem_pavage}.} $x \in \delta_{1/k}(\bar p \cdot C) \in \mathcal T_k$. Whenever $k$ is sufficiently large, $\delta_{1/k}(\bar p \cdot C) \cap K =\emptyset $, hence $h_k(x)=h(x)=0$, the latter implying that also 
     \begin{equation}\label{eq_rrr1}
    h_k \xrightarrow{k \to +\infty} h \equiv 0  \qquad  \leb^{n}\text{-a.e. on }\G \setminus K.
    \end{equation}
  Combining \eqref{eq_rrr} and \eqref{eq_rrr1} and by using Scheffé's lemma (\cite{scheffe}),
\begin{equation}\label{eq_scheffe}
\lim_{k \to +\infty}\nl h_k-h \nr_{L^1} =0.
\end{equation}
 Since $h$ and the $h_k$ are supported in a fixed compact set, \eqref{eq_scheffe} implies
\[
\int_\G |h_k-h|^\frac{Q}{Q+r}\,dx \leq c\left(\int_\G |h_k-h|dx\right)^\frac{Q}{Q+r} \xrightarrow{k\to+\infty}0,
\]
for a constant $c$ depending only on this common compact support. Hence
\[
\|h_k-h\|_{L^\frac{Q}{Q+r}}\xrightarrow{k\to+\infty}0.
\]
Moreover, since $0<\frac{Q}{Q+r}<1$, $|a^\frac{Q}{Q+r}-b^\frac{Q}{Q+r}|\leq |a-b|^\frac{Q}{Q+r}$ for $a,b\geq0$, and therefore
\[
\int_\G |h_k(x)^\frac{Q}{Q+r}-h(x)^\frac{Q}{Q+r}|dx \leq \int_\G |h_k(x)-h(x)|^\frac{Q}{Q+r} dx \xrightarrow{k\to+\infty}0.
\]
and
\[
\|h_k\|_{L^\frac{Q}{Q+r}}\xrightarrow{k\to+\infty} \|h\|_{L^\frac{Q}{Q+r}},
\]
concluding the proof.
\end{proof}

Lemma \ref{lem_step3} below proves Zador's Theorem for compactly supported absolutely continuous (with respect to $\leb^n$) measures. The idea is to use Lemma \ref{lem_convhk} to approximate, in some sense, every measure with a sequence of measures supported on copies of left translation and/or dilations of the Euclidean cube $C$ and use Lemma \ref{lem_step2} to conclude.

\begin{lem}\label{lem_step3}
Let $\mu$ be absolutely continuous with respect to $\leb^{n}$ and with compact support. Then
\[
\lim_{N \to +\infty}N^{\frac{r}{Q}}V_{N,r}(\mu)=Q_r(C)\nl \frac{d\mu}{d\leb^{n}} \nr_{L^{\frac{Q}{Q+r}}}.
\]
\end{lem}
\begin{proof}
Let $K\subset \G$ be a compact set such that $\operatorname{spt}(h)\subset K$, where $h$ is the density of $\mu$ with respect to $\leb^n$. Let $\mathcal P_k$ and $h_k$
be as in Lemma \ref{lem_convhk}. Then
\[
\mu_k\ceq h_k\leb^n
=
\sum_{E\in\mathcal P_k}\mu(E)U(E).
\]
By Lemma \ref{lem_step2} we have, for every $k \in \N$,
\begin{equation}\label{eq_aux1}
\lim_{N\to+\infty}N^\frac{r}{Q}V_{N,r}(\mu_k)=Q_r(C)\|h_k\|_{L^{\frac{Q}{Q+r}}}.
\end{equation}
Fix $0<\sigma<1/2$, and for $N\in\mathbb N$ define
\[
N_1\ceq [ (1-\sigma)N ],
\qquad
N_2\ceq N-N_1.
\]
Choose a compact set $K'\supset \bigcup_{k \in \N} \operatorname{spt}(h_k)$.
Since $K'$ is compact, by Corollary \ref{cor_net}, there exists a constant $A=A(K')>0$ such that, for every $N \in \N$, one can find a finite set
\[
\Lambda_{N_2}\subset K'
\]
satisfying
\[
\#\Lambda_{N_2}\leq N_2,
\qquad
\sup_{x\in K'}d(x,\Lambda_{N_2})\le A N_2^{-\frac{1}{Q}}.
\]
Let $\alpha\in C_{N_1,r}(\mu_k)$ and $\Delta_N\ceq \alpha\cup \Lambda_{N_2}$. We have $\#\Delta_N \leq N$ and 
\begin{align*}
    N^{\frac{r}{Q}}\left| \int \min_{a \in \Delta_N}d(x,a)^r d\mu_k(x) -\int \min_{a \in \Delta_N} d(x,a)^r d\mu(x) \right| &\leq N^{\frac{r}{Q}}\int \min_{a \in \Delta_N} d(x,a)^r |h_k(x)-h(x)|dx\\
    &\leq \left(AN_2^{-\frac{1}{Q}}\right)^r N^{\frac{r}{Q}} \nl h_k -h \nr_{L^1}\\
    &\leq A^r \sigma^{-\frac{r}{Q}} \nl h_k -h \nr_{L^1}
\end{align*}
whenever $N$ is sufficiently large. The latter implies
\begin{align*}
    N^{\frac{r}{Q}}V_{N,r}(\mu) &\leq N^{\frac{r}{Q}}\int \min_{a \in \Delta_N} d(x,a)^r d\mu(x) \\
    &\leq N^{\frac{r}{Q}}\int \min_{a \in \Delta_N} d(x,a)^r d\mu_k(x)+ A^r \sigma^{-\frac{r}{Q}}\nl h_k-h \nr_{L^1}\\
    & \leq N^{\frac{r}{Q}}\int \min_{a \in \alpha} d(x,a)^r d\mu_k(x)+ A^r \sigma^{-\frac{r}{Q}} \nl h_k-h \nr_{L^1}\\
    & =N^{\frac{r}{Q}}V_{N_1,r}(\mu_k)+ A^r \sigma^{-\frac{r}{Q}} \nl h_k-h \nr_{L^1},
\end{align*}
hence
\[
  N^{\frac{r}{Q}}V_{N,r}(\mu) \leq \left(
\frac{N}{N_1}\right)^\frac{r}{Q}N_1^{\frac{r}{Q}}V_{N_1,r}(\mu_k)+ A^r \sigma^{-\frac{r}{Q}} \nl h_k-h \nr_{L^1}.
\]
From \eqref{eq_aux1} we obtain
\[
\limsup_{N \to +\infty }N^{\frac{r}{Q}}V_{N,r}(\mu) \leq (1-\sigma)^{-\frac{r}{Q}}Q_r(C) \nl h_k \nr_{L^{\frac{Q}{Q+r}}}+A^r \sigma^{-\frac{r}{Q}} \nl h_k-h \nr_{L^1}.
\]
By letting $k \to +\infty$ and $\sigma \to 0$ in the above inequality we obtain, thanks to Lemma \ref{lem_convhk},
\[
\limsup_{N \to +\infty }N^{\frac{r}{Q}}V_{N,r}(\mu) \leq Q_r(C) \nl h \nr_{L^{\frac{Q}{Q+r}}}.
\]
We are left to prove 
\[
\liminf_{N \to +\infty }N^{\frac{r}{Q}}V_{N,r}(\mu) \geq Q_r(C) \nl h \nr_{L^{\frac{Q}{Q+r}}}.
\]
Let $\beta \in C_{N_1,r}(\mu)$ and $\tau \ceq \beta \cup \Lambda_{N_2}$. We have $\#\tau \leq N$ and, as before,
\[
N^{\frac{r}{Q}}\left| \int \min_{a \in \tau}d(x,a)^r d\mu_k(x) -\int \min_{a \in \tau} d(x,a)^r d\mu(x) \right| \leq A^r \sigma^{-\frac{r}{Q}} \nl h_k-h \nr_{L^1},
\]
whenever $N$ is sufficiently large. The latter implies
\begin{align*}
    N^{\frac{r}{Q}}V_{N_1,r}(\mu) &= N^{\frac{r}{Q}}\int \min_{a \in \beta} d(x,a)^r d\mu(x) \\
    &\geq N^{\frac{r}{Q}}\int \min_{a \in \tau} d(x,a)^r d\mu(x)\\
    & \geq N^{\frac{r}{Q}}\int \min_{a \in \tau} d(x,a)^r d\mu_k(x)-A^r \sigma^{-\frac{r}{Q}} \nl h_k-h \nr_{L^1}\\
    & \geq N^{\frac{r}{Q}}V_{N,r}(\mu_k)-A^r \sigma^{-\frac{r}{Q}} \nl h_k-h \nr_{L^1},
\end{align*}
hence
\[
  \left(\frac{N}{N_1}\right)^{\frac{r}{Q}}N_1^\frac{r}{Q}V_{N_1,r}(\mu) \geq N^{\frac{r}{Q}}V_{N,r}(\mu_k)- A^r \sigma^{-\frac{r}{Q}} \nl h_k-h \nr_{L^1}.
\]
From \eqref{eq_aux1} we get
\[
(1-\sigma)^{-\frac{r}{Q}}\liminf_{N \to +\infty}N_1^{\frac{r}{Q}}V_{N_1,r}(\mu) \geq Q_r(C)\nl h_k \nr_{L^{\frac{Q}{Q+r}}}- A^r \sigma^{-\frac{r}{Q}} \nl h_k -h \nr_{L^1}.
\]
By letting $k \to +\infty$ and $\sigma \to 0$ in the above inequality we obtain, thanks to Lemma \ref{lem_convhk},
\[
\liminf_{N \to +\infty}N^{\frac{r}{Q}}V_{N,r}(\mu) \geq Q_r(C) \nl h \nr_{L^{\frac{Q}{Q+r}}},
\]
concluding the proof.
\end{proof}

\subsection{Zador's Theorem for compactly supported singular probability measures}
\label{subsec_zad3}

\begin{lem}\label{lem_step4}
 Let $\mu$ be a probability measure which is singular with respect to $\leb^{n}$ and with compact support. Then
\[
\lim_{N \to +\infty}N^{\frac{r}{Q}}V_{N,r}(\mu)=0.
\]
\end{lem}
\begin{proof}
  Let $\ve>0$. Since $\mu$ is compactly supported and singular with respect to $\leb^n$ we can find a bounded open set $O \se \G$ such that 
  \[
  \mu(O)=1 \qquad \text{and}\qquad \leb^n(O) \leq \ve.
  \]
  For every $N \in \N$ we define
  \[
  \lambda_N \ceq \left( \frac{2\ve}{N}\right)^\frac{1}{Q}, \qquad  R_N \ceq \{ \delta_{\lambda_N}(p \cdot C): p \in \Z^n, \delta_{\lambda_N}(p \cdot C) \subset O \}
  \]
  and
  \[
  O_N \ceq \bigsqcup_{E \in R_N}E, \qquad M_N \ceq \#R_N.
  \]
  We observe in passing that $\lambda_N \xrightarrow{N \to +\infty}0$. Since $O_N$ is a disjoint union, by left-invariance and homogeneity, we obtain
  \[
  M_N \lambda_N^Q=\leb^n(O_N) \leq \leb^n(O)  \leq \ve,
  \]
  the latter implying
  \[
  M_N \leq \frac{\ve}{\lambda_N^Q}=\frac{N}{2}.
  \]
  We claim that 
  \begin{equation}\label{eq_claima}
      \mu(O_N) \xrightarrow{N \to +\infty}1.
  \end{equation}
  In fact, let $x \in O$; since $O$ is open, there exists $\rho_x>0$ such that 
  \[
  B(x,\rho_x) \subset O.
  \]
  Let $\bar p \in \Z^n$ such that\footnote{The existence and the uniqueness of such $\bar p$ is guaranteed by Lemma \ref{lem_pavage}.} $x \in \delta_{\lambda_N}(\bar p \cdot C)$ and $D \ceq \operatorname{diam}(C)$. We have
  \[
  \operatorname{diam}(\delta_{\lambda_N}(\bar p \cdot C)) \leq D \lambda_N,
  \]
  and, when $N$ is sufficiently large $D\lambda_N <\rho_x$, the latter implying that 
  \[
  \delta_{\lambda_N}(\bar p \cdot C) \subset O.
  \]
  Consequently $x \in O_N$ whenever $N$ is sufficiently large, i.e,
  \[
  \uno_{O_N}(x) \xrightarrow{N \to +\infty}1 \qquad \text{for every }x \in O.
  \]
  Since $\mu(O)=1$, by dominated convergence we obtain the claim \eqref{eq_claima}.  Now let $K \ceq \operatorname{spt}(\mu)$ and define
  \[
  M'_N \ceq N-M_N.
  \]
    We observe in passing that since $M_N \leq N/2$, then $M_N' \geq N/2$. By Corollary \ref{cor_net} there exist a constant $A=A(K)$ and a finite set $\Lambda_{M_N'} \se K$ such that
  \[
  \#\Lambda_{M_N'}\leq M_N' \qquad \text{ and }\qquad \sup_{x \in K}d(x,\Lambda_{M_N'}) \leq A(M_{N}')^{-\frac{1}{Q}}.
  \]
Moreover, for every $E \in R_N$ we choose a point $p_E \in E$ and we observe that, for every $q \in E$, we have
  \[
  d(q,p_E) \leq D\lambda_N
  \]
We define
\[
\alpha_N \ceq \{ p_E: E \in R_N\} \cup  \Lambda_{M_N'}.
\]
 Since $\#\alpha_N \leq N$ and we have
 \[
 \begin{cases}
d(x,\alpha_N) \leq D\lambda_N, \qquad x \in O_N\\
d(x,\alpha_N) \leq A(M_N')^{-\frac{1}{Q}}, \qquad x \in K \setminus O_N
 \end{cases}
 \]
 we infer
 \[
 V_{N,r}(\mu) \leq D^r\lambda_N^r\mu(O_N)+A^r(M_N')^{-\frac{r}{Q}}\mu(K \setminus O_N)
 \]
 multiplying the above inequality by $N^\frac{r}{Q}$ and recalling that, by the definitions of $\lambda_N$ and $M_N'$, we have
 \[
 N^\frac{r}{Q}\lambda_N^r=(2\ve)^\frac{r}{Q} \qquad \text{ and }\qquad \left(\frac{N}{M_N'}\right)^\frac{r}{Q} \leq 2^\frac{r}{Q},
 \]
 we obtain
 \[
 N^\frac{r}{Q}V_{N,r}(\mu) \leq D^r (2\ve)^\frac{r}{Q}+A^r2^\frac{r}{Q}\mu(K \setminus O_N).
 \]
 By taking the $\limsup$ of the above inequality and recalling that $\mu(O_N) \xrightarrow{N \to +\infty}1$ we obtain
 \[
 \limsup_{N \to +\infty}N^\frac{r}{Q}V_{N,r}(\mu) \leq D^r(2\ve)^\frac{r}{Q}.
 \]
 Letting $\ve \to 0$ is enough to conclude the proof.
\end{proof}

\begin{cor}\label{cor_compact_zador}
Let $\nu$ be a compactly supported probability measure on $\G$. Let $\nu=\nu_{\mathrm{ac}}+\nu_{\mathrm{s}}$ be its Lebesgue decomposition and set
\[
h\ceq \frac{d\nu_{\mathrm{ac}}}{d\leb^n}.
\]
Then
\[
\lim_{N\to+\infty}N^{\frac rQ}V_{N,r}(\nu)=Q_r(C)\|h\|_{L^{\frac{Q}{Q+r}}}.
\]
\end{cor}
\begin{proof}
If $\nu_{\mathrm{ac}}(\G)=0$, the conclusion follows from Lemma \ref{lem_step4}. If $\nu_{\mathrm{s}}(\G)=0$, it follows from Lemma \ref{lem_step3}. Otherwise set
\[
a\ceq \nu_{\mathrm{ac}}(\G)\in(0,1),\qquad
\nu_1\ceq \frac{\nu_{\mathrm{ac}}}{a},\qquad
\nu_2\ceq \frac{\nu_{\mathrm{s}}}{1-a}.
\]
By Lemmas \ref{lem_step3} and \ref{lem_step4},
\[
\lim_{N\to+\infty}N^{\frac rQ}V_{N,r}(\nu_1)
=Q_r(C)\left\|\frac{h}{a}\right\|_{L^{\frac{Q}{Q+r}}},
\qquad
\lim_{N\to+\infty}N^{\frac rQ}V_{N,r}(\nu_2)=0.
\]
Since $\nu=a\nu_1+(1-a)\nu_2$, Lemma \ref{lem_basicpro}(v.B) gives
\[
\lim_{N\to+\infty}N^{\frac rQ}V_{N,r}(\nu)
=aQ_r(C)\left\|\frac{h}{a}\right\|_{L^{\frac{Q}{Q+r}}}
=Q_r(C)\|h\|_{L^{\frac{Q}{Q+r}}},
\]
where in the last equality we used homogeneity of the $L^{\frac{Q}{Q+r}}$ quasi-norm.
\end{proof}

\subsection{Zador's Theorem for probability measures in Carnot groups}
\label{subsec_zad4}

The following estimate (Lemma \ref{lem_stepaux}) is true for every probability measure with finite $r$-moment. A higher moment assumption is needed only for the corresponding $\limsup$ inequality, where Pierce's lemma controls the tail.

\begin{lem}\label{lem_stepaux}
Let $\mu=\mu_{\mathrm{ac}}+\mu_{\mathrm{s}}$ be the Lebesgue decomposition of $\mu$, and set $h\ceq d\mu_{\mathrm{ac}}/d\leb^n$. We have
    \[
    \liminf_{N \to +\infty}N^{\frac{r}{Q}}V_{N,r}(\mu) \geq Q_r(C) \nl h \nr_{L^{\frac{Q}{Q+r}}}.
    \]
\end{lem}
\begin{proof}
Define, for $k \in \N$, the dilated Euclidean cube $C_k \ceq \delta_k(C)$. Since $C_k\uparrow \G$, we have $\mu(C_k)>0$ for all sufficiently large $k$; for such $k$ set $\mu_k\ceq \frac{\mu\res C_k}{\mu(C_k)}$. By Corollary \ref{cor_compact_zador},
    \begin{equation}\label{eq_bbb}
    \lim_{N \to +\infty} N^{\frac{r}{Q}}V_{N,r}(\mu_k)=Q_r(C) \nl h_k \nr_{L^{\frac{Q}{Q+r}}}, 
    \end{equation}
    where 
    \[
   h_k \ceq \frac{d(\mu_k)_{\mathrm{ac}}}{d\leb^n}=\frac{\uno_{C_k}}{\mu(C_k)}h.
    \]
    By Lemma \ref{lem_basicpro}, if $\mu(C_k)<1$ we write $\mu=\mu(C_k)\mu_k+(1-\mu(C_k))\mu_k^c$ for a probability measure $\mu_k^c$, while if $\mu(C_k)=1$ the following inequality is immediate; we have $V_{N,r}(\mu) \geq \mu(C_k)V_{N,r}(\mu_k)$, hence
    \[
    \liminf_{N \to +\infty} N^{\frac{r}{Q}}V_{N,r}(\mu) \geq Q_r(C) \nl h \uno_{C_k} \nr_{L^{\frac{Q}{Q+r}}}.
    \]
    Letting $k \to +\infty$ and using the monotone convergence theorem is enough to conclude.
\end{proof}

We are finally ready to prove Theorem \ref{teo_intro}.
\begin{proof}[Proof of Theorem~\ref{teo_intro}]
 By Lemma \ref{lem_qpositivo}, $Q_r(C)>0$. Let $\mu=\mu_{\mathrm{ac}}+\mu_{\mathrm{s}}$ be the Lebesgue decomposition of $\mu$, and set $h\ceq d\mu_{\mathrm{ac}}/d\leb^n$. Lemma \ref{lem_stepaux} gives the desired $\liminf$ inequality. Define for $k \in \N$ the dilated Euclidean cube $C_k \ceq \delta_{k}(C)$. If $\mu(C_k^c)=0$ for some $k$, then $\mu$ is compactly supported and the conclusion follows from Corollary \ref{cor_compact_zador}. We may therefore assume that $\mu(C_k^c)>0$ for every $k$. For all sufficiently large $k$, also $\mu(C_k)>0$; for such $k$ define
    \[
    \mu_k \ceq \frac{\mu \res C_k}{\mu(C_k)}, \qquad
    \tilde{\mu}_k \ceq \frac{\mu \res C_k^c}{\mu(C_k^c)}.
    \]
    Moreover, for such $k$, we decompose $\mu$ as
    \[
\mu=\mu(C_k)\mu_k+\mu(C_k^c)\tilde{\mu}_k.  
    \]
    Fix $0<\ve<1$; from Corollary \ref{cor_compact_zador} applied to $\mu_k$, whose absolutely continuous density is $\uno_{C_k}h/\mu(C_k)$, and Lemma \ref{lem_basicpro} we obtain 
    \begin{equation}\label{eq_bbb1}
    \limsup_{N \to +\infty} N^{\frac{r}{Q}}V_{N,r}(\mu) \leq (1-\ve)^{-\frac{r}{Q}}Q_r(C) \nl h \uno_{C_k} \nr_{L^{\frac{Q}{Q+r}}}+\mu(C_k^c)\ve^{-\frac{r}{Q}}\limsup_{N \to +\infty}N^{\frac{r}{Q}}V_{N,r}(\tilde{\mu}_k).
    \end{equation}
    Thanks to Theorem \ref{teo_pierce}, we have
  \begin{equation}\label{eq_bbb2}
    \mu(C_k^c)N^{\frac{r}{Q}}V_{N,r}(\tilde{\mu}_k) \leq c  \left(\mu(C_k^c)+\int_{C_k^c} d(x,0)^{r+\delta}d\mu(x)\right)
    \end{equation}
    whenever $N \geq 2$. Since $\mu$ has finite $(r+\delta)$-moment, combining \eqref{eq_bbb1} and \eqref{eq_bbb2} and letting $k \to +\infty$ and then $\ve \to 0$ is enough to obtain
    \[
    \limsup_{N \to +\infty}N^{\frac{r}{Q}}V_{N,r}(\mu) \leq Q_r(C)\nl h \nr_{L^{\frac{Q}{Q+r}}}.
    \]
    Combining the latter with Lemma \ref{lem_stepaux} is enough to conclude.
\end{proof}

\subsection{Empirical measures of optimal centers}

We now prove the second main result. The argument combines the localized lower bound used above with the equality case in the allocation inequality.

\begin{proof}[Proof of Theorem~\ref{teo_empirical_centers}]
Let
\[
h\ceq \frac{d\mu_{\mathrm{ac}}}{d\leb^n},
\]
and set
\[
q\ceq \frac{Q}{Q+r},\qquad
Z_h\ceq \int_\G h^q\,d\leb^n .
\]
The equality $\#\alpha_N=N$ follows from Lemma \ref{lem_card}, since $\mu_{\mathrm{ac}}(\G)>0$ implies that the support of $\mu$ is infinite. Moreover $0<Z_h<+\infty$: the positivity follows from $\mu_{\mathrm{ac}}(\G)>0$, while the finiteness follows from H\"older's inequality, the finite $(r+\delta)$-moment assumption, and the volume growth of homogeneous balls. Indeed,
\[
\int_\G h^q\,d\leb^n
\leq
\left(\int_\G h(x)(1+d(x,0)^{r+\delta})\,d\leb^n(x)\right)^q
\left(\int_\G (1+d(x,0)^{r+\delta})^{-\frac{Q}{r}}\,d\leb^n(x)\right)^{1-q}
<+\infty .
\]
Let
\[
\eta\ceq \frac{h^q}{Z_h}\,\leb^n .
\]
We shall prove that $\nu_N(A)\to\eta(A)$ for every Borel set $A\subset\G$ such that $\eta(\partial A)=0$; this is enough to conclude by the standard characterization of weak convergence through continuity sets.

Fix such a set $A$. Let $(N_j)_j$ be a subsequence such that
\[
\nu_{N_j}(A)\to v\in[0,1].
\]
We first obtain the localized lower bound associated with this subsequence. Let $K$ and $L$ be compact sets compactly contained in $\operatorname{int}(A)$ and $\operatorname{int}(A^c)$, respectively. If $K=\emptyset$ we set $\Gamma_K=\emptyset$; otherwise, since $K$ is compactly contained in $A$, there exists a finite set $\Gamma_K\subset K$ such that
\[
d(x,\Gamma_K)\leq d(x,\G\setminus A)
\qquad\text{for every }x\in K.
\]
Similarly, if $L=\emptyset$ we set $\Gamma_L=\emptyset$; otherwise there exists a finite set $\Gamma_L\subset L$ such that
\[
d(x,\Gamma_L)\leq d(x,A)
\qquad\text{for every }x\in L.
\]
Set
\[
N_A(j)\ceq \#(\alpha_{N_j}\cap A),
\qquad
N_{A^c}(j)\ceq \#(\alpha_{N_j}\cap A^c).
\]
Then $N_A(j)/N_j\to v$ and $N_{A^c}(j)/N_j\to 1-v$. For $x\in K$, every center in $\alpha_{N_j}\setminus A$ is at distance at least $d(x,\G\setminus A)$, hence
\[
d(x,\alpha_{N_j})\geq d\bigl(x,(\alpha_{N_j}\cap A)\cup\Gamma_K\bigr).
\]
The analogous inequality holds on $L$, with $A^c$ and $\Gamma_L$ in place of $A$ and $\Gamma_K$. Therefore, writing $\mu_K\ceq\mu\res K/\mu(K)$ and $\mu_L\ceq\mu\res L/\mu(L)$ whenever the normalizing constants are positive, and understanding the corresponding term as zero when the normalizing constant is zero, we obtain
\begin{align*}
V_{N_j,r}(\mu)
&=\int_\G d(x,\alpha_{N_j})^r\,d\mu(x)\geq
\mu(K)V_{N_A(j)+\#\Gamma_K,r}(\mu_K)
+
\mu(L)V_{N_{A^c}(j)+\#\Gamma_L,r}(\mu_L).
\end{align*}
We now apply Lemma \ref{lem_stepaux} to $\mu_K$ and $\mu_L$; their absolutely continuous densities are \(h\uno_K/\mu(K)\) and \(h\uno_L/\mu(L)\), whenever the corresponding masses are positive. For example, if $\mu(K)>0$ and $v>0$, setting $m_j\ceq N_A(j)+\#\Gamma_K$, we have
\[
N_j^{\frac rQ}\mu(K)V_{m_j,r}(\mu_K)
=
\mu(K)\left(\frac{N_j}{m_j}\right)^{\frac rQ}
m_j^{\frac rQ}V_{m_j,r}(\mu_K),
\]
and therefore, since $m_j/N_j\to v$,
\[
\liminf_{j\to+\infty}N_j^{\frac rQ}\mu(K)V_{m_j,r}(\mu_K)
\geq
Q_r(C)\nl h\uno_K\nr_{L^q}v^{-\frac rQ}.
\]
Combining this estimate with the analogous estimate for the compact set $L$, and using the homogeneity of the $L^q$ quasi-norm after multiplying by the masses $\mu(K)$ and $\mu(L)$, gives
\[
\liminf_{j\to+\infty}N_j^{\frac rQ}V_{N_j,r}(\mu)
\geq
Q_r(C)\left(
\nl h\uno_K\nr_{L^q}v^{-\frac rQ}
+
\nl h\uno_L\nr_{L^q}(1-v)^{-\frac rQ}
\right).
\]

Since $\eta(\partial A)=0$, by the inner regularity of $h^q\leb^n$, we can take increasing compact exhaustions of $\operatorname{int}(A)$ and $\operatorname{int}(A^c)$ and, using monotone convergence, obtain
\[
\liminf_{j\to+\infty}N_j^{\frac rQ}V_{N_j,r}(\mu)
\geq
Q_r(C)\left(
\nl h\uno_A\nr_{L^q}v^{-\frac rQ}
+
\nl h\uno_{A^c}\nr_{L^q}(1-v)^{-\frac rQ}
\right).
\]
On the other hand, Theorem \ref{teo_intro} gives
\[
\lim_{j\to+\infty}N_j^{\frac rQ}V_{N_j,r}(\mu)
=Q_r(C)\nl h\nr_{L^q}.
\]
Since $Q_r(C)>0$, we infer
\[
\nl h\uno_A\nr_{L^q}v^{-\frac rQ}
+
\nl h\uno_{A^c}\nr_{L^q}(1-v)^{-\frac rQ}
\leq
\nl h\nr_{L^q}.
\]
If
\[
I_A\ceq\int_A h^q\,d\leb^n,
\qquad
I_{A^c}\ceq\int_{A^c}h^q\,d\leb^n,
\]
the latter inequality is
\[
I_A^{1/q}v^{-\frac rQ}
+
I_{A^c}^{1/q}(1-v)^{-\frac rQ}
\leq
(I_A+I_{A^c})^{1/q}.
\]
The reverse inequality follows from \cite[Lemma 6.8]{gl00}, equivalently from H\"older's inequality in the allocation form used in Lemma \ref{lem_step2}. More explicitly, if both \(I_A\) and \(I_{A^c}\) are positive and \(\alpha\ceq 1/q\), the function
\[
t\mapsto I_A^\alpha t^{1-\alpha}
+I_{A^c}^\alpha(1-t)^{1-\alpha},
\qquad t\in(0,1),
\]
has the unique minimum \((I_A+I_{A^c})^\alpha\) at \(t=I_A/(I_A+I_{A^c})\). Hence equality holds only when
\[
v=\frac{I_A}{I_A+I_{A^c}}=\eta(A),
\]
with the convention that, if $I_A=0$, the equality condition gives $v=0$, while if $I_{A^c}=0$, it gives $v=1$. Hence every convergent subsequence of $\nu_N(A)$ has limit $\eta(A)$, and therefore $\nu_N(A)\to\eta(A)$. This holds for every $\eta$-continuity set $A$, so $\nu_N\rightharpoonup\eta$.
\end{proof}

\bibliographystyle{acm}
\bibliography{qombib}

@book{gl00,
 author = {Graf, Siegfried and Luschgy, Harald},
 title = {Foundations of quantization for probability distributions},
 fseries = {Lecture Notes in Mathematics},
 series = {Lect. Notes Math.},
 issn = {0075-8434},
 volume = {1730},
 isbn = {3-540-67394-6},
 year = {2000},
 publisher = {Berlin: Springer},
 language = {English},
 doi = {10.1007/BFb0103945},
 zbMATH = {1437844},
 Zbl = {0951.60003}
}

@article {z82,
    AUTHOR = {Zador, Paul L.},
     TITLE = {Asymptotic quantization error of continuous signals and the
              quantization dimension},
   JOURNAL = {IEEE Trans. Inform. Theory},
  FJOURNAL = {Institute of Electrical and Electronics Engineers.
              Transactions on Information Theory},
    VOLUME = {28},
      YEAR = {1982},
    NUMBER = {2},
     PAGES = {139--149},
      ISSN = {0018-9448,1557-9654},
   MRCLASS = {94A05},
  MRNUMBER = {651809},
       DOI = {10.1109/TIT.1982.1056490},
       URL = {https://doi.org/10.1109/TIT.1982.1056490},
}

@article {bw82,
    AUTHOR = {Bucklew, James A. and Wise, Gary L.},
     TITLE = {Multidimensional asymptotic quantization theory with {$r$}th
              power distortion measures},
   JOURNAL = {IEEE Trans. Inform. Theory},
  FJOURNAL = {Institute of Electrical and Electronics Engineers.
              Transactions on Information Theory},
    VOLUME = {28},
      YEAR = {1982},
    NUMBER = {2},
     PAGES = {239--247},
      ISSN = {0018-9448,1557-9654},
   MRCLASS = {94A34 (60G35)},
  MRNUMBER = {651819},
       DOI = {10.1109/TIT.1982.1056486},
       URL = {https://doi.org/10.1109/TIT.1982.1056486},
}

@article {g04,
    AUTHOR = {Gruber, Peter M.},
     TITLE = {Optimum quantization and its applications},
   JOURNAL = {Adv. Math.},
  FJOURNAL = {Advances in Mathematics},
    VOLUME = {186},
      YEAR = {2004},
    NUMBER = {2},
     PAGES = {456--497},
      ISSN = {0001-8708,1090-2082},
   MRCLASS = {94A24 (52A27 52B60)},
  MRNUMBER = {2073915},
MRREVIEWER = {Ahmed\ I.\ Zayed},
       DOI = {10.1016/j.aim.2003.07.017},
       URL = {https://doi.org/10.1016/j.aim.2003.07.017},
}

@article {k12,
    AUTHOR = {Kloeckner, Beno\^it},
     TITLE = {Approximation by finitely supported measures},
   JOURNAL = {ESAIM Control Optim. Calc. Var.},
  FJOURNAL = {ESAIM. Control, Optimisation and Calculus of Variations},
    VOLUME = {18},
      YEAR = {2012},
    NUMBER = {2},
     PAGES = {343--359},
      ISSN = {1292-8119,1262-3377},
   MRCLASS = {49Q20 (41A60 60B05)},
  MRNUMBER = {2954629},
MRREVIEWER = {Berardo\ Ruffini},
       DOI = {10.1051/cocv/2010100},
       URL = {https://doi.org/10.1051/cocv/2010100},
}

@article {i16,
    AUTHOR = {Iacobelli, Mikaela},
     TITLE = {Asymptotic quantization for probability measures on
              {R}iemannian manifolds},
   JOURNAL = {ESAIM Control Optim. Calc. Var.},
  FJOURNAL = {ESAIM. Control, Optimisation and Calculus of Variations},
    VOLUME = {22},
      YEAR = {2016},
    NUMBER = {3},
     PAGES = {770--785},
      ISSN = {1292-8119,1262-3377},
   MRCLASS = {60D99 (49Q20)},
  MRNUMBER = {3527943},
       DOI = {10.1051/cocv/2015025},
       URL = {https://doi.org/10.1051/cocv/2015025},
}

@article {ai25,
    AUTHOR = {Ayd{\i}n, Ata Deniz and Iacobelli, Mikaela},
     TITLE = {Asymptotic quantization of measures on {R}iemannian manifolds
              via covering growth estimates},
   JOURNAL = {Adv. Math.},
  FJOURNAL = {Advances in Mathematics},
    VOLUME = {474},
      YEAR = {2025},
     PAGES = {Paper No. 110311, 43},
      ISSN = {0001-8708,1090-2082},
   MRCLASS = {53D50 (46T12 49Q22 53C23)},
  MRNUMBER = {4902658},
       DOI = {10.1016/j.aim.2025.110311},
       URL = {https://doi.org/10.1016/j.aim.2025.110311},
}

@article {a25,
    AUTHOR = {Ayd{\i}n, Ata Deniz},
     TITLE = {Asymptotics of the quantization problem on metric measure
              spaces},
   JOURNAL = {Math. Ann.},
  FJOURNAL = {Mathematische Annalen},
    VOLUME = {394},
      YEAR = {2026},
    NUMBER = {2},
     PAGES = {48},
      ISSN = {0025-5831,1432-1807},
   MRCLASS = {99-06},
  MRNUMBER = {5031624},
       DOI = {10.1007/s00208-026-03376-x},
       URL = {https://doi.org/10.1007/s00208-026-03376-x},
}

@article {p70,
    AUTHOR = {Pierce, John N.},
     TITLE = {Asymptotic quantizing error for unbounded random variables},
   JOURNAL = {IEEE Trans. Information Theory},
  FJOURNAL = {Institute of Electrical and Electronics Engineers.
              Transactions on Information Theory},
    VOLUME = {IT-16},
      YEAR = {1970},
     PAGES = {81--83},
      ISSN = {0018-9448,1557-9654},
   MRCLASS = {60.30},
  MRNUMBER = {272030},
       DOI = {10.1109/tit.1970.1054400},
       URL = {https://doi.org/10.1109/tit.1970.1054400},
}

@article {bpt14,
    AUTHOR = {Bourne, D. P. and Peletier, M. A. and Theil, F.},
     TITLE = {Optimality of the triangular lattice for a particle system
              with {W}asserstein interaction},
   JOURNAL = {Comm. Math. Phys.},
  FJOURNAL = {Communications in Mathematical Physics},
    VOLUME = {329},
      YEAR = {2014},
    NUMBER = {1},
     PAGES = {117--140},
      ISSN = {0010-3616,1432-0916},
   MRCLASS = {82D25 (82D60)},
  MRNUMBER = {3206999},
       DOI = {10.1007/s00220-014-1965-5},
       URL = {https://doi.org/10.1007/s00220-014-1965-5},
}

@article {bs09,
    AUTHOR = {Buttazzo, Giuseppe and Santambrogio, Filippo},
     TITLE = {A mass transportation model for the optimal planning of an
              urban region},
   JOURNAL = {SIAM Rev.},
  FJOURNAL = {SIAM Review},
    VOLUME = {51},
      YEAR = {2009},
    NUMBER = {3},
     PAGES = {593--610},
      ISSN = {1095-7200,0036-1445},
   MRCLASS = {49Q20 (49J45 90B06)},
  MRNUMBER = {2535060},
MRREVIEWER = {Luigi\ De Pascale},
       DOI = {10.1137/090759197},
       URL = {https://doi.org/10.1137/090759197},
}

@article {bjm11,
    AUTHOR = {Bouchitt\'e, Guy and Jimenez, Chlo\'e{} and Mahadevan, Rajesh},
     TITLE = {Asymptotic analysis of a class of optimal location problems},
   JOURNAL = {J. Math. Pures Appl. (9)},
  FJOURNAL = {Journal de Math\'ematiques Pures et Appliqu\'ees. Neuvi\`eme
              S\'erie},
    VOLUME = {95},
      YEAR = {2011},
    NUMBER = {4},
     PAGES = {382--419},
      ISSN = {0021-7824,1776-3371},
   MRCLASS = {49Q20 (49J45)},
  MRNUMBER = {2776375},
MRREVIEWER = {Giandomenico\ Orlandi},
       DOI = {10.1016/j.matpur.2010.10.009},
       URL = {https://doi.org/10.1016/j.matpur.2010.10.009},
}

@book {s15,
    AUTHOR = {Santambrogio, Filippo},
     TITLE = {Optimal transport for applied mathematicians},
    SERIES = {Progress in Nonlinear Differential Equations and their
              Applications},
    VOLUME = {87},
      NOTE = {Calculus of variations, PDEs, and modeling},
 PUBLISHER = {Birkh\"auser/Springer, Cham},
      YEAR = {2015},
     PAGES = {xxvii+353},
      ISBN = {978-3-319-20827-5; 978-3-319-20828-2},
   MRCLASS = {49-02 (35J96 49J45 49M29 58E50 90C05 90C48 91B02)},
  MRNUMBER = {3409718},
MRREVIEWER = {Luigi\ De Pascale},
       DOI = {10.1007/978-3-319-20828-2},
       URL = {https://doi.org/10.1007/978-3-319-20828-2},
}

@article {fssc03,
    AUTHOR = {Franchi, Bruno and Serapioni, Raul and Serra Cassano,
              Francesco},
     TITLE = {On the structure of finite perimeter sets in step 2 {C}arnot
              groups},
   JOURNAL = {J. Geom. Anal.},
  FJOURNAL = {The Journal of Geometric Analysis},
    VOLUME = {13},
      YEAR = {2003},
    NUMBER = {3},
     PAGES = {421--466},
      ISSN = {1050-6926,1559-002X},
   MRCLASS = {49Q15 (53C17)},
  MRNUMBER = {1984849},
MRREVIEWER = {J.\ E.\ Brothers},
       DOI = {10.1007/BF02922053},
       URL = {https://doi.org/10.1007/BF02922053},
}

@book {abb,
    AUTHOR = {Agrachev, Andrei and Barilari, Davide and Boscain, Ugo},
     TITLE = {A comprehensive introduction to sub-{R}iemannian geometry},
    SERIES = {Cambridge Studies in Advanced Mathematics},
    VOLUME = {181},
      NOTE = {From the Hamiltonian viewpoint,
              With an appendix by Igor Zelenko},
 PUBLISHER = {Cambridge University Press, Cambridge},
      YEAR = {2020},
     PAGES = {xviii+745},
      ISBN = {978-1-108-47635-5},
   MRCLASS = {53C17},
  MRNUMBER = {3971262},
MRREVIEWER = {Luca Rizzi},
}

@book {montgomery,
    AUTHOR = {Montgomery, Richard},
     TITLE = {A tour of subriemannian geometries, their geodesics and
              applications},
    SERIES = {Mathematical Surveys and Monographs},
    VOLUME = {91},
 PUBLISHER = {American Mathematical Society, Providence, RI},
      YEAR = {2002},
     PAGES = {xx+259},
      ISBN = {0-8218-1391-9},
   MRCLASS = {53C17 (37J99 53C60 58E10 70G45 70H05)},
  MRNUMBER = {1867362},
MRREVIEWER = {Andrey\ V.\ Sarychev},
       DOI = {10.1090/surv/091},
       URL = {https://doi.org/10.1090/surv/091},
}

@book {cdptbook,
    AUTHOR = {Capogna, Luca and Danielli, Donatella and Pauls, Scott D. and
              Tyson, Jeremy T.},
     TITLE = {An introduction to the {H}eisenberg group and the
              sub-{R}iemannian isoperimetric problem},
    SERIES = {Progress in Mathematics},
    VOLUME = {259},
 PUBLISHER = {Birkh\"auser Verlag, Basel},
      YEAR = {2007},
     PAGES = {xvi+223},
      ISBN = {978-3-7643-8132-5; 3-7643-8132-9},
   MRCLASS = {53C17 (22E30 30C65 32T27 32V15 49Q15)},
  MRNUMBER = {2312336},
MRREVIEWER = {Piotr\ Haj\l asz},
       DOI = {10.1007/978-3-7643-8133-2},
       URL = {https://doi.org/10.1007/978-3-7643-8133-2},
}

@book{LDlibro,
 author = {Le Donne, Enrico},
 title = {Metric {Lie} groups. {Carnot}-{Carath{\'e}odory} spaces from the homogeneous viewpoint},
 fseries = {Graduate Texts in Mathematics},
 series = {Grad. Texts Math.},
 issn = {0072-5285},
 volume = {306},
 isbn = {978-3-031-98831-8; 978-3-031-98834-9; 978-3-031-98832-5},
 year = {2025},
 publisher = {Cham: Springer},
 language = {English},
 zbMATH = {8074076}
}

@article {ld17,
    AUTHOR = {Le Donne, Enrico},
     TITLE = {A primer on {C}arnot groups: homogenous groups,
              {C}arnot-{C}arath\'eodory spaces, and regularity of their
              isometries},
   JOURNAL = {Anal. Geom. Metr. Spaces},
  FJOURNAL = {Analysis and Geometry in Metric Spaces},
    VOLUME = {5},
      YEAR = {2017},
    NUMBER = {1},
     PAGES = {116--137},
      ISSN = {2299-3274},
   MRCLASS = {53C17 (22E25 22F30 43A80)},
  MRNUMBER = {3742567},
MRREVIEWER = {Andrea\ Pinamonti},
       DOI = {10.1515/agms-2017-0007},
       URL = {https://doi.org/10.1515/agms-2017-0007},
}

@article {b48,
    AUTHOR = {Bennett, W. R.},
     TITLE = {Spectra of quantized signals},
   JOURNAL = {Bell System Tech. J.},
  FJOURNAL = {The Bell System Technical Journal},
    VOLUME = {27},
      YEAR = {1948},
     PAGES = {446--472},
      ISSN = {0005-8580},
   MRCLASS = {60.0X},
  MRNUMBER = {26287},
MRREVIEWER = {J.\ L.\ Doob},
       DOI = {10.1002/j.1538-7305.1948.tb01340.x},
       URL = {https://doi.org/10.1002/j.1538-7305.1948.tb01340.x},
}

@article {ops48,
       author = {{Oliver}, B.~M. and {Pierce}, J.~R. and {Shannon}, C.~E.},
        title = "{The Philosophy of PCM}",
      journal = {Proceedings of the IRE},
         year = {1948},
        month = {nov},
       volume = {36},
       number = {11},
        pages = {1324-1331},
          doi = {10.1109/JRPROC.1948.231941},
       adsurl = {https://ui.adsabs.harvard.edu/abs/1948PIRE...36.1324O}
}

@article {g79,
    AUTHOR = {Gersho, Allen},
     TITLE = {Asymptotically optimal block quantization},
   JOURNAL = {IEEE Trans. Inform. Theory},
  FJOURNAL = {Institute of Electrical and Electronics Engineers.
              Transactions on Information Theory},
    VOLUME = {25},
      YEAR = {1979},
    NUMBER = {4},
     PAGES = {373--380},
      ISSN = {0018-9448,1557-9654},
   MRCLASS = {94A34},
  MRNUMBER = {536229},
       DOI = {10.1109/TIT.1979.1056067},
       URL = {https://doi.org/10.1109/TIT.1979.1056067},
}

@book{gg92,
 author = {Gersho, Allen and Gray, Robert M.},
 title = {Vector quantization and signal compression},
 isbn = {0-7923-9181-0},
 year = {1992},
 publisher = {Boston, MA: Kluwer Academic Publishers},
 language = {English},
 zbMATH = {467196},
 Zbl = {0782.94001}
}

@article{gn02,
  title={Quantization},
  author={Gray, Robert M. and Neuhoff, David L.},
  journal={IEEE transactions on information theory},
  volume={44},
  number={6},
  pages={2325--2383},
  year={2002},
  publisher={IEEE}
}

@incollection {g15,
    AUTHOR = {Pag\`es, Gilles},
     TITLE = {Introduction to vector quantization and its applications for
              numerics},
 BOOKTITLE = {C{EMRACS} 2013---modelling and simulation of complex systems:
              stochastic and deterministic approaches},
    SERIES = {ESAIM Proc. Surveys},
    VOLUME = {48},
     PAGES = {29--79},
 PUBLISHER = {EDP Sci., Les Ulis},
      YEAR = {2015},
   MRCLASS = {60G35 (62E17 65C05 81S05 94A15)},
  MRNUMBER = {3415387},
MRREVIEWER = {Ramon\ van Handel},
       DOI = {10.1051/proc/201448002},
       URL = {https://doi.org/10.1051/proc/201448002},
}

@article {bc21,
    AUTHOR = {Bourne, David P. and Cristoferi, Riccardo},
     TITLE = {Asymptotic optimality of the triangular lattice for a class of
              optimal location problems},
   JOURNAL = {Comm. Math. Phys.},
  FJOURNAL = {Communications in Mathematical Physics},
    VOLUME = {387},
      YEAR = {2021},
    NUMBER = {3},
     PAGES = {1549--1602},
      ISSN = {0010-3616,1432-0916},
   MRCLASS = {49Q20 (82B10 90B80)},
  MRNUMBER = {4324385},
       DOI = {10.1007/s00220-021-04216-6},
       URL = {https://doi.org/10.1007/s00220-021-04216-6},
}

@article {cgi15,
    AUTHOR = {Caglioti, Emanuele and Golse, Fran\c cois and Iacobelli,
              Mikaela},
     TITLE = {A gradient flow approach to quantization of measures},
   JOURNAL = {Math. Models Methods Appl. Sci.},
  FJOURNAL = {Mathematical Models and Methods in Applied Sciences},
    VOLUME = {25},
      YEAR = {2015},
    NUMBER = {10},
     PAGES = {1845--1885},
      ISSN = {0218-2025,1793-6314},
   MRCLASS = {35K59 (35B40 35K92 94A12)},
  MRNUMBER = {3358447},
MRREVIEWER = {Xiaoliu\ Wang},
       DOI = {10.1142/S0218202515500475},
       URL = {https://doi.org/10.1142/S0218202515500475},
}

@article {cgi18,
    AUTHOR = {Caglioti, Emanuele and Golse, Fran\c cois and Iacobelli,
              Mikaela},
     TITLE = {Quantization of probability distributions and gradient flows
              in space dimension 2},
   JOURNAL = {Ann. Inst. H. Poincar\'e{} C Anal. Non Lin\'eaire},
  FJOURNAL = {Annales de l'Institut Henri Poincar\'e{} C. Analyse Non
              Lin\'eaire},
    VOLUME = {35},
      YEAR = {2018},
    NUMBER = {6},
     PAGES = {1531--1555},
      ISSN = {0294-1449,1873-1430},
   MRCLASS = {35K40 (35B40 35K59 35K92 35Q94 94A12)},
  MRNUMBER = {3846235},
       DOI = {10.1016/j.anihpc.2017.12.003},
       URL = {https://doi.org/10.1016/j.anihpc.2017.12.003},
}

@incollection {i18,
    AUTHOR = {Iacobelli, Mikaela},
     TITLE = {A gradient flow perspective on the quantization problem},
 BOOKTITLE = {P{DE} models for multi-agent phenomena},
    SERIES = {Springer INdAM Ser.},
    VOLUME = {28},
     PAGES = {145--165},
 PUBLISHER = {Springer, Cham},
      YEAR = {2018},
      ISBN = {978-3-030-01946-4; 978-3-030-01947-1},
   MRCLASS = {28A33 (53D25 81P15)},
  MRNUMBER = {3888971},
}

@article {i19,
    AUTHOR = {Iacobelli, Mikaela},
     TITLE = {Asymptotic analysis for a very fast diffusion equation arising
              from the 1{D} quantization problem},
   JOURNAL = {Discrete Contin. Dyn. Syst.},
  FJOURNAL = {Discrete and Continuous Dynamical Systems},
    VOLUME = {39},
      YEAR = {2019},
    NUMBER = {9},
     PAGES = {4929--4943},
      ISSN = {1078-0947,1553-5231},
   MRCLASS = {35K57 (35B40 35K65 70F45)},
  MRNUMBER = {3986316},
MRREVIEWER = {Hiroki\ Hoshino},
       DOI = {10.3934/dcds.2019201},
       URL = {https://doi.org/10.3934/dcds.2019201},
}

@article {ips19,
    AUTHOR = {Iacobelli, Mikaela and Patacchini, Francesco S. and
              Santambrogio, Filippo},
     TITLE = {Weighted ultrafast diffusion equations: from well-posedness to
              long-time behaviour},
   JOURNAL = {Arch. Ration. Mech. Anal.},
  FJOURNAL = {Archive for Rational Mechanics and Analysis},
    VOLUME = {232},
      YEAR = {2019},
    NUMBER = {3},
     PAGES = {1165--1206},
      ISSN = {0003-9527,1432-0673},
   MRCLASS = {35K59 (35B30 35B40)},
  MRNUMBER = {3928748},
MRREVIEWER = {Hongwei\ Chen},
       DOI = {10.1007/s00205-018-01341-w},
       URL = {https://doi.org/10.1007/s00205-018-01341-w},
}

@article {bjr02,
    AUTHOR = {Bouchitt\'e, Guy and Jimenez, Chlo\'e{} and Rajesh, Mahadevan},
     TITLE = {Asymptotique d'un probl\`eme de positionnement optimal},
   JOURNAL = {C. R. Math. Acad. Sci. Paris},
  FJOURNAL = {Comptes Rendus Math\'ematique. Acad\'emie des Sciences. Paris},
    VOLUME = {335},
      YEAR = {2002},
    NUMBER = {10},
     PAGES = {853--858},
      ISSN = {1631-073X,1778-3569},
   MRCLASS = {49J45 (90B30 90B85)},
  MRNUMBER = {1947712},
       DOI = {10.1016/S1631-073X(02)02575-X},
       URL = {https://doi.org/10.1016/S1631-073X(02)02575-X},
}

@article {br15,
    AUTHOR = {Bourne, D. P. and Roper, S. M.},
     TITLE = {Centroidal power diagrams, {L}loyd's algorithm, and
              applications to optimal location problems},
   JOURNAL = {SIAM J. Numer. Anal.},
  FJOURNAL = {SIAM Journal on Numerical Analysis},
    VOLUME = {53},
      YEAR = {2015},
    NUMBER = {6},
     PAGES = {2545--2569},
      ISSN = {0036-1429,1095-7170},
   MRCLASS = {65K10 (49M05 65D99 90B85)},
  MRNUMBER = {3419889},
MRREVIEWER = {Gerd\ Wachsmuth},
       DOI = {10.1137/141000993},
       URL = {https://doi.org/10.1137/141000993},
}

@article {qvg99,
    AUTHOR = {Du, Qiang and Faber, Vance and Gunzburger, Max},
     TITLE = {Centroidal {V}oronoi tessellations: applications and
              algorithms},
   JOURNAL = {SIAM Rev.},
  FJOURNAL = {SIAM Review},
    VOLUME = {41},
      YEAR = {1999},
    NUMBER = {4},
     PAGES = {637--676},
      ISSN = {1095-7200,0036-1445},
   MRCLASS = {52B55 (52C22 65D30 68U05)},
  MRNUMBER = {1722997},
MRREVIEWER = {Ren\ Ding},
       DOI = {10.1137/S0036144599352836},
       URL = {https://doi.org/10.1137/S0036144599352836},
}

@article {qw05,
    AUTHOR = {Du, Qiang and Wang, Desheng},
     TITLE = {The optimal centroidal {V}oronoi tessellations and the
              {G}ersho's conjecture in the three-dimensional space},
   JOURNAL = {Comput. Math. Appl.},
  FJOURNAL = {Computers \& Mathematics with Applications. An International
              Journal},
    VOLUME = {49},
      YEAR = {2005},
    NUMBER = {9-10},
     PAGES = {1355--1373},
      ISSN = {0898-1221,1873-7668},
   MRCLASS = {65D18 (52B55 52C22 68U07 94A20)},
  MRNUMBER = {2149486},
MRREVIEWER = {H.\ W.\ Guggenheimer},
       DOI = {10.1016/j.camwa.2004.12.008},
       URL = {https://doi.org/10.1016/j.camwa.2004.12.008},
}

@incollection {m67,
    AUTHOR = {MacQueen, J.},
     TITLE = {Some methods for classification and analysis of multivariate
              observations},
 BOOKTITLE = {Proc. {F}ifth {B}erkeley {S}ympos. {M}ath. {S}tatist. and
              {P}robability ({B}erkeley, {C}alif., 1965/66), {V}ol. {I}:
              {S}tatistics},
     PAGES = {281--297},
 PUBLISHER = {Univ. California Press, Berkeley, CA},
      YEAR = {1967},
   MRCLASS = {62.40},
  MRNUMBER = {214227},
MRREVIEWER = {K.\ G.\ J\"oreskog},
}

@article {c18,
    AUTHOR = {Chevallier, Julien},
     TITLE = {Uniform decomposition of probability measures: quantization,
              clustering and rate of convergence},
   JOURNAL = {J. Appl. Probab.},
  FJOURNAL = {Journal of Applied Probability},
    VOLUME = {55},
      YEAR = {2018},
    NUMBER = {4},
     PAGES = {1037--1045},
      ISSN = {0021-9002,1475-6072},
   MRCLASS = {60E15 (60B10 60F99 62E17)},
  MRNUMBER = {3899926},
       DOI = {10.1017/jpr.2018.69},
       URL = {https://doi.org/10.1017/jpr.2018.69},
}

@article {k20,
    AUTHOR = {Kloeckner, Beno\^it R.},
     TITLE = {Empirical measures: regularity is a counter-curse to
              dimensionality},
   JOURNAL = {ESAIM Probab. Stat.},
  FJOURNAL = {ESAIM. Probability and Statistics},
    VOLUME = {24},
      YEAR = {2020},
     PAGES = {408--434},
      ISSN = {1292-8100,1262-3318},
   MRCLASS = {60B10 (49Q20 60J05 62E17)},
  MRNUMBER = {4153634},
       DOI = {10.1051/ps/2019025},
       URL = {https://doi.org/10.1051/ps/2019025},
}

@inproceedings{mss21,
  title={Non-asymptotic convergence bounds for Wasserstein approximation using point clouds},
  author={Quentin M{\'e}rigot and Filippo Santambrogio and Cl{\'e}ment Sarrazin},
  booktitle={Neural Information Processing Systems},
  year={2021},
  url={https://api.semanticscholar.org/CorpusID:235436063}
}

@article {hs82,
    AUTHOR = {Hochbaum, Dorit and Steele, J. Michael},
     TITLE = {Steinhaus's geometric location problem for random samples in
              the plane},
   JOURNAL = {Adv. in Appl. Probab.},
  FJOURNAL = {Advances in Applied Probability},
    VOLUME = {14},
      YEAR = {1982},
    NUMBER = {1},
     PAGES = {56--67},
      ISSN = {0001-8678,1475-6064},
   MRCLASS = {60D05 (52A40 90B05)},
  MRNUMBER = {644008},
MRREVIEWER = {Noel\ Cressie},
       DOI = {10.2307/1426733},
       URL = {https://doi.org/10.2307/1426733},
}

@article {z85,
    AUTHOR = {Zemel, Eitan},
     TITLE = {Probabilistic analysis of geometric location problems},
   JOURNAL = {SIAM J. Algebraic Discrete Methods},
  FJOURNAL = {Society for Industrial and Applied Mathematics. Journal on
              Algebraic and Discrete Methods},
    VOLUME = {6},
      YEAR = {1985},
    NUMBER = {2},
     PAGES = {189--200},
      ISSN = {0196-5212},
   MRCLASS = {90B10 (68Q25)},
  MRNUMBER = {777998},
       DOI = {10.1137/0606017},
       URL = {https://doi.org/10.1137/0606017},
}

@article {gl02,
    AUTHOR = {Graf, Siegfried and Luschgy, Harald},
     TITLE = {Rates of convergence for the empirical quantization error},
   JOURNAL = {Ann. Probab.},
  FJOURNAL = {The Annals of Probability},
    VOLUME = {30},
      YEAR = {2002},
    NUMBER = {2},
     PAGES = {874--897},
      ISSN = {0091-1798,2168-894X},
   MRCLASS = {60F15 (60E15 62H30 94A29)},
  MRNUMBER = {1905859},
MRREVIEWER = {Ludger\ R\"uschendorf},
       DOI = {10.1214/aop/1023481010},
       URL = {https://doi.org/10.1214/aop/1023481010},
}

@article {dss13,
    AUTHOR = {Dereich, Steffen and Scheutzow, Michael and Schottstedt, Reik},
     TITLE = {Constructive quantization: approximation by empirical
              measures},
   JOURNAL = {Ann. Inst. Henri Poincar\'e{} Probab. Stat.},
  FJOURNAL = {Annales de l'Institut Henri Poincar\'e{} Probabilit\'es et
              Statistiques},
    VOLUME = {49},
      YEAR = {2013},
    NUMBER = {4},
     PAGES = {1183--1203},
      ISSN = {0246-0203,1778-7017},
   MRCLASS = {60F25 (60B10 65C50 68P30 94A12)},
  MRNUMBER = {3127919},
MRREVIEWER = {A.\ F.\ Gualtierotti},
       DOI = {10.1214/12-AIHP489},
       URL = {https://doi.org/10.1214/12-AIHP489},
}

@article {ts15,
    AUTHOR = {Garc\'ia Trillos, Nicol\'as and Slep\v{c}ev, Dejan},
     TITLE = {On the rate of convergence of empirical measures in
              {$\infty$}-transportation distance},
   JOURNAL = {Canad. J. Math.},
  FJOURNAL = {Canadian Journal of Mathematics. Journal Canadien de
              Math\'ematiques},
    VOLUME = {67},
      YEAR = {2015},
    NUMBER = {6},
     PAGES = {1358--1383},
      ISSN = {0008-414X,1496-4279},
   MRCLASS = {60B10 (05C70 05D40 60D05)},
  MRNUMBER = {3415656},
MRREVIEWER = {Hac\`ene\ Djellout},
       DOI = {10.4153/CJM-2014-044-6},
       URL = {https://doi.org/10.4153/CJM-2014-044-6},
}

@article {ag19,
    AUTHOR = {Ambrosio, Luigi and Glaudo, Federico},
     TITLE = {Finer estimates on the 2-dimensional matching problem},
   JOURNAL = {J. \'Ec. polytech. Math.},
  FJOURNAL = {Journal de l'\'Ecole polytechnique. Math\'ematiques},
    VOLUME = {6},
      YEAR = {2019},
     PAGES = {737--765},
      ISSN = {2429-7100,2270-518X},
   MRCLASS = {60D05 (49J55 49Q20 60H15)},
  MRNUMBER = {4014635},
       DOI = {10.5802/jep.105},
       URL = {https://doi.org/10.5802/jep.105},
}

@article {agt19,
    AUTHOR = {Ambrosio, Luigi and Glaudo, Federico and Trevisan, Dario},
     TITLE = {On the optimal map in the 2-dimensional random matching
              problem},
   JOURNAL = {Discrete Contin. Dyn. Syst.},
  FJOURNAL = {Discrete and Continuous Dynamical Systems},
    VOLUME = {39},
      YEAR = {2019},
    NUMBER = {12},
     PAGES = {7291--7308},
      ISSN = {1078-0947,1553-5231},
   MRCLASS = {60D05 (35B35 35F21 49J55 49Q20 58J35)},
  MRNUMBER = {4026190},
       DOI = {10.3934/dcds.2019304},
       URL = {https://doi.org/10.3934/dcds.2019304},
}

@article {ast19,
    AUTHOR = {Ambrosio, Luigi and Stra, Federico and Trevisan, Dario},
     TITLE = {A {PDE} approach to a 2-dimensional matching problem},
   JOURNAL = {Probab. Theory Related Fields},
  FJOURNAL = {Probability Theory and Related Fields},
    VOLUME = {173},
      YEAR = {2019},
    NUMBER = {1-2},
     PAGES = {433--477},
      ISSN = {0178-8051,1432-2064},
   MRCLASS = {60D05 (49J55 60H15)},
  MRNUMBER = {3916111},
       DOI = {10.1007/s00440-018-0837-x},
       URL = {https://doi.org/10.1007/s00440-018-0837-x},
}

@article {bc20,
    AUTHOR = {Benedetto, Dario and Caglioti, Emanuele},
     TITLE = {Euclidean random matching in 2{D} for non-constant densities},
   JOURNAL = {J. Stat. Phys.},
  FJOURNAL = {Journal of Statistical Physics},
    VOLUME = {181},
      YEAR = {2020},
    NUMBER = {3},
     PAGES = {854--869},
      ISSN = {0022-4715,1572-9613},
   MRCLASS = {60D05 (82B44)},
  MRNUMBER = {4160913},
MRREVIEWER = {Dominique\ Jeulin},
       DOI = {10.1007/s10955-020-02608-x},
       URL = {https://doi.org/10.1007/s10955-020-02608-x},
}

@article {bccdss21,
    AUTHOR = {Benedetto, D. and Caglioti, E. and Caracciolo, S. and
              D'Achille, M. and Sicuro, G. and Sportiello, A.},
     TITLE = {Random assignment problems on {$2d$} manifolds},
   JOURNAL = {J. Stat. Phys.},
  FJOURNAL = {Journal of Statistical Physics},
    VOLUME = {183},
      YEAR = {2021},
    NUMBER = {2},
     PAGES = {Paper No. 34, 40},
      ISSN = {0022-4715,1572-9613},
   MRCLASS = {49Q22 (49J55 58J50 60G99)},
  MRNUMBER = {4257835},
       DOI = {10.1007/s10955-021-02768-4},
       URL = {https://doi.org/10.1007/s10955-021-02768-4},
}

@article {g01,
    AUTHOR = {Gruber, Peter M.},
     TITLE = {Optimal configurations of finite sets in {R}iemannian
              2-manifolds},
   JOURNAL = {Geom. Dedicata},
  FJOURNAL = {Geometriae Dedicata},
    VOLUME = {84},
      YEAR = {2001},
    NUMBER = {1-3},
     PAGES = {271--320},
      ISSN = {0046-5755,1572-9168},
   MRCLASS = {52C05 (52A10 52B10)},
  MRNUMBER = {1825361},
MRREVIEWER = {Alan\ S.\ McRae},
       DOI = {10.1023/A:1010358407868},
       URL = {https://doi.org/10.1023/A:1010358407868},
}

@article {lbp19,
    AUTHOR = {Le Brigant, Alice and Puechmorel, St\'ephane},
     TITLE = {Quantization and clustering on {R}iemannian manifolds with an
              application to air traffic analysis},
   JOURNAL = {J. Multivariate Anal.},
  FJOURNAL = {Journal of Multivariate Analysis},
    VOLUME = {173},
      YEAR = {2019},
     PAGES = {685--703},
      ISSN = {0047-259X,1095-7243},
   MRCLASS = {62H30 (53Z05 62H35 62P30)},
  MRNUMBER = {3959747},
       DOI = {10.1016/j.jmva.2019.05.008},
       URL = {https://doi.org/10.1016/j.jmva.2019.05.008},
}

@article{lbp19bis,
AUTHOR = {Le Brigant, Alice and Puechmorel, Stéphane},
TITLE = {Approximation of Densities on Riemannian Manifolds},
JOURNAL = {Entropy},
VOLUME = {21},
YEAR = {2019},
NUMBER = {1},
ARTICLE-NUMBER = {43},
URL = {https://www.mdpi.com/1099-4300/21/1/43},
PubMedID = {33266759},
ISSN = {1099-4300},
DOI = {10.3390/e21010043}
}

@article {scheffe,
    AUTHOR = {Scheff\'e, Henry},
     TITLE = {A useful convergence theorem for probability distributions},
   JOURNAL = {Ann. Math. Statistics},
  FJOURNAL = {Annals of Mathematical Statistics},
    VOLUME = {18},
      YEAR = {1947},
     PAGES = {434--438},
      ISSN = {0003-4851},
   MRCLASS = {27.2X},
  MRNUMBER = {21585},
MRREVIEWER = {R.\ Fortet},
       DOI = {10.1214/aoms/1177730390},
       URL = {https://doi.org/10.1214/aoms/1177730390},
}

@article {fgn05,
    AUTHOR = {Franchi, Bruno and Guti\'errez, Cristian E. and van Nguyen,
              Truyen},
     TITLE = {Homogenization and convergence of correctors in {C}arnot
              groups},
   JOURNAL = {Comm. Partial Differential Equations},
  FJOURNAL = {Communications in Partial Differential Equations},
    VOLUME = {30},
      YEAR = {2005},
    NUMBER = {10-12},
     PAGES = {1817--1841},
      ISSN = {0360-5302,1532-4133},
   MRCLASS = {35B27 (22E25 35A30 35H10 49N60)},
  MRNUMBER = {2182313},
MRREVIEWER = {Taras\ A.\ Mel\cprime nyk},
       DOI = {10.1080/03605300500300014},
       URL = {https://doi.org/10.1080/03605300500300014},
}

@article {ddmm20,
    AUTHOR = {Dirr, Nicolas and Dragoni, Federica and Mannucci, Paola and
              Marchi, Claudio},
     TITLE = {{$\Gamma$}-convergence and homogenisation for a class of
              degenerate functionals},
   JOURNAL = {Nonlinear Anal.},
  FJOURNAL = {Nonlinear Analysis. Theory, Methods \& Applications. An
              International Multidisciplinary Journal},
    VOLUME = {190},
      YEAR = {2020},
     PAGES = {111618, 25},
      ISSN = {0362-546X,1873-5215},
   MRCLASS = {49J45},
  MRNUMBER = {4021977},
MRREVIEWER = {Micol\ Amar},
       DOI = {10.1016/j.na.2019.111618},
       URL = {https://doi.org/10.1016/j.na.2019.111618},
}

@article {ft02,
    AUTHOR = {Franchi, Bruno and Tesi, Maria Carla},
     TITLE = {Two-scale homogenization in the {H}eisenberg group},
   JOURNAL = {J. Math. Pures Appl. (9)},
  FJOURNAL = {Journal de Math\'ematiques Pures et Appliqu\'ees. Neuvi\`eme
              S\'erie},
    VOLUME = {81},
      YEAR = {2002},
    NUMBER = {6},
     PAGES = {495--532},
      ISSN = {0021-7824},
   MRCLASS = {35B27 (35J70 43A80)},
  MRNUMBER = {1912410},
MRREVIEWER = {Alain\ Brillard},
       DOI = {10.1016/S0021-7824(01)01247-8},
       URL = {https://doi.org/10.1016/S0021-7824(01)01247-8},
}

@article {bw03,
    AUTHOR = {Birindelli, I. and Wigniolle, J.},
     TITLE = {Homogenization of {H}amilton-{J}acobi equations in the
              {H}eisenberg group},
   JOURNAL = {Commun. Pure Appl. Anal.},
  FJOURNAL = {Communications on Pure and Applied Analysis},
    VOLUME = {2},
      YEAR = {2003},
    NUMBER = {4},
     PAGES = {461--479},
      ISSN = {1534-0392,1553-5258},
   MRCLASS = {35B27 (35H20)},
  MRNUMBER = {2019062},
MRREVIEWER = {A.\ K.\ Nandakumaran},
       DOI = {10.3934/cpaa.2003.2.461},
       URL = {https://doi.org/10.3934/cpaa.2003.2.461},
}

\end{document}